\documentclass[12pt]{amsart}
\usepackage{amssymb,latexsym,amsmath,amscd}
\usepackage[latin1]{inputenc}
\usepackage[usenames,dvipsnames]{color}
\usepackage[all]{xy}
\usepackage{hyperref}

\newcommand{\PP}{\mathbb{P}}
\newcommand{\ZZ}{\mathbb{Z}}

\newcommand{\OO}{\mathcal{O}}

\newcommand{\LL}{\mathcal{L}}
\newcommand{\MM}{\mathcal{M}}
\newcommand{\NN}{\mathcal{N}}
\newcommand{\EE}{\mathcal{E}}
\newcommand{\FF}{\mathcal{F}}
\newcommand{\GG}{\mathcal{G}}

\newcommand{\JJ}{\mathcal{J}}

\newcommand{\lra}{\longrightarrow}
\newcommand{\ra}{\rightarrow}

\newcommand{\al}{\alpha}

\newcommand{\Hom}{\textrm{Hom}}
\renewcommand{\ker}{\operatorname{Ker}\,}

\newcommand{\srel}{\stackrel}
\renewcommand{\deg}{\operatorname{deg}}

\newcommand{\rg}{\operatorname{rk}}

\newcommand{\depth}{\operatorname{depth}}
\newcommand{\End}{\operatorname{End}}
\newcommand{\Ext}{\operatorname{Ext}}
\newcommand{\length}{\operatorname{length}}
\newcommand{\Pic}{\operatorname{Pic}}

\newcommand{\im}{\operatorname{Im}}

\theoremstyle{definition}
\newtheorem{defi}{Definition}[section]
\theoremstyle{plain}
\newtheorem{lema}[defi]{Lemma}
\newtheorem{teo}[defi]{Theorem}

\newtheorem{prop}[defi]{Proposition}

\newtheorem{cor}[defi]{Corollary}
\newtheorem{rk}[defi]{Remark}

\newtheorem{ex}[defi]{Example}

\newtheorem{teo-def}[defi]{Theorem/Definition}

\newtheorem{theorem}{Theorem}[section]
\newtheorem{proposition}[theorem]{Proposition}

\theoremstyle{remark}
\newtheorem{remark}[theorem]{Remark}

\title[]{Stable Ulrich bundles}
\author[]{Marta Casanellas}
\address{Departament de Matemàtica Aplicada I. ETSEIB. Universitat Polit\`ecnica de Catalunya. Avinguda Diagonal 647. 08028 Barcelona. Spain.}
\thanks{Research of the first author partially supported by Ministerio de Educaci\'on y Ciencia, MTM2009-14163-
C02-02, and Generalitat de Catalunya, 2009 SGR 1284.}
\email{marta.casanellas@upc.edu}
\author[]{Robin Hartshorne}
\address{Department of Mathematics. Evans Hall. University of California. Berkeley, CA, 94720-3840. USA}
\email{robin@math.berkeley.edu}
\begin{document}

\maketitle
\begin{center}
With an Appendix by Florian Gei{\ss}   and
Frank-Olaf Schreyer
\end{center}

\begin{abstract}
The existence of stable ACM vector bundles of high rank on
algebraic varieties is a challenging problem. In this paper, we
study stable Ulrich bundles (that is, stable ACM bundles whose
corresponding module has the maximum number of generators) on
nonsingular cubic surfaces $X \subset \mathbb{P}^3.$ We give
necessary and sufficient conditions on the first Chern class $D$
for the existence of stable Ulrich bundles on $X$ of rank $r$ and
$c_1=D$. When such bundles exist, we prove that that the
corresponding moduli space of stable bundles is smooth and
irreducible of dimension $D^2-2r^2+1$ and consists entirely of
stable Ulrich bundles (see Theorem 1.1). As a consequence, we are
also able to prove the existence of stable Ulrich bundles of any
rank on nonsingular cubic threefolds in $\mathbb{P}^4$, and we
show that the restriction map from bundles on the threefold to
bundles on the surface is generically étale and dominant.
\end{abstract}

\maketitle

\section{Introduction}
The study of moduli spaces of stable vector bundles of given rank
and Chern classes on algebraic varieties is a very active topic in
algebraic geometry. See for example the book \cite{HL}. In recent
years attention has focused on ACM bundles, that is vector bundles
without intermediate cohomology. Recently ACM bundles on
hypersurfaces have been used to provide counterexamples to a
conjecture of Griffiths and Harris about whether subvarieties of
codimension two of a hypersurface can be obtained by intersecting
with a subvariety of codimension two of the ambient space
\cite{KRR2}. There have been numerous studies of rank 2 ACM
bundles on surfaces and threefolds (see \cite{Beauvillecubic},
\cite{AM}, \cite{ChiantiniMadonna2}, \cite{Madonna2000},
\cite{ChiantiniFaenzi}, \cite{BrambillaFaenzi}, and the references
in those papers), and a few studies of ACM bundles of higher rank
\cite{ArrondoGrana}, \cite{AM}, \cite{Madonna2005}. There have
also been a few examples of indecomposable ACM bundles of
arbitrarily high rank \cite{MiroRoigPons},
\cite{PonsLlopisTonini}. But as far as we can tell, examples of
stable ACM bundles of higher ranks are essentially unknown.

In our earlier paper \cite{CH2} we began such a study by
constructing stable ACM bundles of all ranks  $r$ on a nonsingular
cubic surface, with first Chern class $rH$. The bundles we
constructed are of a particular kind,  the so-called
\textit{Ulrich bundles}: that is ACM bundles whose corresponding
module has the maximum number of generators. We refer to the
introduction of our earlier paper \cite{CH2} for the history and
motivation for considering these concepts and the corresponding
notions of maximal Cohen-Macaulay modules and Ulrich modules in
local algebra.

In this paper we continue that study by determining for which
divisor classes on the cubic surface there are Ulrich bundles or
stable Ulrich bundles. The following theorem summarizes our main
results on cubic surfaces.

\begin{teo}{\rm (Theorem \ref{teo_stable})}
Let $D$ be a divisor on a nonsingular cubic surface $X \subset
\PP^3,$ and let $r\geq 2$ be an integer. Then there exist stable
Ulrich bundles $\EE$ of rank $r$ with $c_1(\EE)=D$ if and only if
$0\leq D.L \leq 2r$ for all lines $L$ on $X$, and $D.T \geq 2r$
for all twisted cubic curves $T$ on $X$, with one exception.

Moreover, if $D$ satisfies the conditions above, the moduli space
$\MM_X^s(r;c_1,c_2)$ of stable vector bundles on $X$ of rank $r$,
$c_1=D$ and $c_2=\frac{D^2-r}{2}$, is smooth and irreducible of
dimension $D^2-2r^2+1$ and consists entirely of stable Ulrich
vector bundles.
\end{teo}

In particular, if $D=rH$, where $H$ is the hyperplane class, this
gives a new proof of the main theorem (5.7) of \cite{CH2}.

Using this result we are also able to prove the existence of
stable Ulrich bundles of any rank on any nonsingular cubic
threefold $Y\subset \PP^4$.

\begin{teo}{\rm (see Theorem \ref{teo_3fold} and Corollary \ref{finalCorollary})} For any $r\geq 2$, the moduli space of stable rank $r$ Ulrich bundles on a general cubic threefold
$Y$ in $\PP^4$ is non-empty and smooth of dimension $r^2+1.$
Furthermore, it has an open subset for which restriction to a
hyperplane section gives an étale dominant map to the moduli of
stable bundles on the cubic surface.
\end{teo}

The motivation for writing this paper was to clarify the proof of
stability for orientable (i.e. $c_1 = rH$)  Ulrich bundles given
in our previous paper \cite[5.3]{CH2}. We thank R. M. Mir\'o-Roig
for pointing out that the proof of \cite[5.3]{CH2} was not clear
enough.

In section 2 we prove some generalities about Ulrich bundles. The
most important result in this section shows that Ulrich bundles of
any rank on nonsingular projective varieties of any dimension are
semistable (cf. Theorem \ref{teo_semist}). In this section we also
discuss modular families of simple Ulrich bundles, which will be
crucial to show the existence of stable bundles. Indeed, our
technique for proving Theorem 1.1 is first to show the existence
of simple Ulrich bundles. For these, we can compute the dimension
of modular families by deformation theory. Then we show that the
non-stable simple bundles form a family of smaller dimension, so
the general bundle of that modular family must be stable.

In section 3 we study Ulrich bundles on a nonsingular cubic
surface $X$. We first prove that, if they exist, they form an
irreducible family. Then in Theorem \ref{condD} we give necessary
and sufficient conditions on the first Chern class D for the
existence of some Ulrich bundle of rank $r$ with this first Chern
class. This requires a careful analysis of the intersection
properties of $D$ with lines and twisted cubic curves on the
surface.

In section 4 we prove our main theorem giving necessary and
sufficient conditions on $c_1(\EE)$ for the existence of a
\textit{stable} Ulrich bundle $\EE$ (see Theorem
\ref{teo_stable}).  As a consequence we recover Faenzi's results
\cite{Faenzi} on stable rank 2 Ulrich bundles. As an illustration
of our results we give the classification of Ulrich bundles of
rank 3 on $X$.

In the last section we consider nonsingular cubic 3-folds and
prove the existence of stable Ulrich bundles of all ranks
(Theorem \ref{teo_3fold}). We would like to thank Florian Gei{\ss} and Frank-Olaf Schreyer
for computations in \texttt{Macaulay2} that provide essential ingredients for extending
our results from the cubic surface to cubic threefolds (see Appendix).

We expect that our results will generalize naturally using the
same ideas to other del Pezzo surfaces and Fano threefolds. We
have restricted our attention to the cubic surface and the cubic
threefold for simplicity. What might be more interesting would be
to explore the existence of higher rank stable Ulrich bundles on
surfaces in $\PP^3$  and threefolds in $\PP^4$  of higher degree.
Here is what is known about the existence of Ulrich bundles of
rank $r \geq 2$ on a general hypersurface $X$ of degree $d \geq 3$
and dimension $N \geq 2$ in $\PP^{N+1}$. It follows from the work
of Beauville \cite{Beauville} that rank 2 Ulrich bundles exist for
$N=2$ if and only if $d \leq 15$; and for $N = 3$ if and only if
$d \leq 5.$ For $N\geq 4$ and $d \geq 3$ there are no rank 2
Ulrich bundles \cite{ChiantiniMadonna2}. On the other hand, a
theorem of Herzog, Ulrich and Backelin \cite{HerzogUlrichBackelin}
shows that every hypersuface admits an Ulrich bundle of some high
rank. A simple calculation with Chern classes (see Remark
\ref{Rmk_even}) shows that $r(d-1)$ must be even for the existence
of an Ulrich bundle, so the first open cases would seem to be
$N=2$, $r=3,$ $d=5;$ and $N=3, r=4, d=6$. See also \cite{KRR},
\cite{Madonna2000}, \cite{Madonna2005}, \cite{AM},
\cite{BrambillaFaenzi}, \cite{ChiantiniFaenzi} for further details
on Ulrich bundles and more general ACM bundles.

Throughout this paper we work over an algebraically closed field
$k$ of arbitrary characteristic.

\section{Generalities on Ulrich sheaves on projective varieties}

In this section we review the definition and cohomological
properties of Ulrich bundles. We show that Ulrich bundles on any
nonsingular projective variety are semistable, and if stable they
are also $\mu$-stable. We discuss the two kinds of moduli spaces
we will use in the paper. And for future reference we compute
$\chi(\EE \otimes \FF^{\vee})$ for Ulrich bundles $\EE, \FF$ on an
algebraic surface.

Let $X$ be an integral projective variety with a fixed very ample
invertible sheaf $\OO_X(1)$, over an algebraically closed field
$k$ of arbitrary characteristic. Let $d$ be the degree of $X$ in
the embedding defined by $\OO_X(1)$. A coherent sheaf $\EE$ on $X$
is \textit{arithmetically Cohen-Macaulay} (briefly \textit{ACM})
if $\EE$ is locally Cohen-Macaulay, and
$H^i_{\ast}(\EE):=\oplus_{t\in \ZZ}H^i(\EE(t))=0$ for $0<i<dim X$.
If $\EE$ is an ACM sheaf of rank $r$, one knows that the number
$m(\EE)$ of generators of the graded module $H^0_{\ast}(\EE)$ is
$\leq dr$ \cite[3.1]{CH2}. In our previous paper \cite{CH2} we
defined $\EE$ to be an Ulrich sheaf if the number of generators of
$H^0_{\ast}(\EE)$ achieved this maximum. To simplify terminology,
in this paper we give a  more restrictive definition:

\begin{defi}
Let $X, \OO_X(1)$ be an algebraic variety of degree $d$. We say an
ACM sheaf $\EE$ on $X$ of rank $r$ is an \textit{Ulrich sheaf} if
the module $H^0_{\ast}(\EE)$ has the maximum number of generators
$dr$, and the generators are all in degree 0. (Thus an Ulrich
sheaf is what we called a ``normalized Ulrich sheaf generated by
global sections'' in \cite{CH2}.) If $X$ is a nonsingular variety
then an Ulrich sheaf is always locally free, so we may call it a
vector bundle.
\end{defi}

\begin{lema}\label{lemaH0}
If $\EE$ is an ACM sheaf on an integral projective variety $X$ of
degree $d$, with $h^0(\EE(-1))=0$ and $h^0(\EE)=dr$, then $\EE$ is
Ulrich.
\end{lema}

\begin{proof}
$h^0(\EE(-1))=0$ implies that $h^0(\EE)\leq m(\EE)$ which  is
$\leq dr$ by \cite[3.1]{CH2}. Then equality $h^0(\EE)=dr$ implies
$m(\EE)=dr$ and the generators are all in degree zero, so $\EE$ is
Ulrich.
\end{proof}

\begin{prop}\label{prop1curves}
Let $X,\OO_X(1)$ be a nonsingular projective curve of degree $d$
and genus $g$, and let $\EE$ be a rank $r$ locally free sheaf on
$X$. Assume $h^0(\EE(-1))=0$. Then
\begin{itemize}
\item [(a)] $h^0(\EE) \leq dr$
\item [(b)] $\deg \EE \leq r(d+g-1)$
\item [(c)] $\chi(\EE(n))\leq dr(n+1)$ for all $n \in \ZZ$.
\end{itemize}
Furthermore, equality in any one of {\rm (a), (b), (c)} implies
equality in the other two, and is equivalent to $\EE$ being an
Ulrich sheaf.
\end{prop}

\begin{proof} This result follows from \cite[section 4]{EisSchreyer}
but for convenience we include a self-contained proof. First of
all (a) follows from \cite[Theorem 3.1]{CH2}, since any locally
free sheaf on a curve is ACM. Applying the Riemann-Roch theorem to
$\EE(-1)$ and using the hypothesis $h^0(\EE(-1))=0$, we find
$$\chi(\EE(-1))=\deg \EE -rd+r(1-g) \leq 0,$$
which gives (b). Then Riemann-Roch for $\EE(n)$ says
$$\chi(\EE(n))=\deg \EE+nrd+r(1-g),$$
and substituting (b) gives (c).

Note that equality in (b) is equivalent to equality in (c).
Equality in (b) implies $\chi(\EE(-1))=0$, so $H^1(\EE(-1))=0$.
Hence from the sequence $0 \lra \EE(-1) \lra \EE \lra \EE_H\lra 0$
where $H$ is a general hyperplane section, consisting of $d$
points, we see that $h^0(\EE)= h^0(\EE_H) =dr$, which gives
equality in (a).

Conversely, equality in (a), using the same exact sequence, shows
that $\alpha:H^1(\EE(-1))\lra H^1(\EE)$ is injective. Since the
map $H^0(\EE(n))\lra H^0(\EE_H(n))$ will also be surjective for $n
\geq 0$, we find $\alpha(n):H^1(\EE(n-1))\lra H^1(\EE(n))$ is also
injective for $n\geq 0$. So by Serre's vanishing theorem,
$H^1(\EE(n))=0$ for $n>>0$ and hence for all $n\geq -1$. Therefore
$\chi(\EE(-1))=0$, and we get equality in (b), equivalent to
equality in (c).

Moreover, equality in (a) is equivalent to $\EE$ being Ulrich.
Indeed, if $\EE$ is Ulrich, then it has all its generators in
degree 0 and therefore $h^0(\EE)=dr$. Conversely, if we have
equality in (a), since we have assumed $h^0(\EE(-1))$, $\EE$ is
Ulrich by Lemma \ref{lemaH0}.
\end{proof}

The following lemma implies that the definition of Ulrich sheaf
given in this paper and the one given in \cite{EisSchreyer}
coincide.

\begin{lema}\label{lemaHn} Let $\EE$ be an Ulrich bundle
of rank $r$ on a nonsingular projective variety $X$ of dimension
$N$ and $\deg X=d$. Then,
\begin{itemize}
 \item[(i)] The restriction $\EE_H$ of $\EE$ to a general hyperplane section is also an Ulrich bundle.
\item[(ii)] $h^0(\EE)=dr$ and $H^N(\EE(i))=0$ for any $i \geq -N$.
\item[(iii)] $\deg \EE= r(d+g-1)$ 
$g$ is the genus of $X \cap H^{N-1}$ for a general hyperplane
section $H$.
\item[(iv)] If furthermore $X$ is subcanonical with
$\omega_X=\OO_X(m)$ for some $m\in \ZZ,$ then the twisted dual
sheaf $\EE^{\vee}(N+m+1)$ is also an Ulrich bundle.
\end{itemize}
\end{lema}

\begin{proof} If $N=1$, then (i), (ii), and (iii) follow from Proposition
\ref{prop1curves}. For (iv) note that $h^0(\EE^{\vee}(m+2))$ is
equal to $h^1(\EE(-2))$ by duality. But $h^0(\EE(-2))=0$ and
$\chi(\EE(-2))=-dr$ by Proposition \ref{prop1curves}, so
$h^0(\EE^{\vee}(m+2))=dr,$ which shows that $\EE^{\vee}(m+2)$ is
Ulrich.

For $N\geq 2$ we use induction on $N$. By Bertini's theorem, we
may assume that a general hyperplane section $H$ is also
nonsingular.

(i) Consider the exact sequence
$$0 \lra \EE(-1) \lra \EE \lra \EE_H\lra 0$$
for a general hyperplane section $H$. It is easy to see that
$\EE_H$ is an ACM sheaf on $H$. Furthermore, from the exact
sequence we see that $H^0(\EE_H(-1))=0$ and $h^0(\EE_H)=dr$, where
$d=\deg(X)$ and $r=\rg \EE$, so that $\EE_H$ is an Ulrich sheaf on
$H$ by Lemma \ref{lemaH0}.

(ii) Since $\EE_H$ is Ulrich by (i), the induction hypothesis
applies to $\EE_H$ so $H^{N-1}(\EE_H(t))=0$ for $t \geq -(N-1)$
and $h^0(\EE_H)=dr$. Therefore $h^0(\EE)=dr$ and $H^N(\EE(t-1))
\cong H^N(\EE(t))$ for all $t\geq -N+1$. By Serre's vanishing, all
these higher cohomologies are 0 and we are done.

(iii) Since $\deg \EE$ coincides with $\deg \EE_H$ for a general
hyperplane section $H$ of $X$ and $\EE_H$ is Ulrich by (i), we can
apply the induction hypothesis to it. Therefore we obtain $\deg
\EE=r(d+g-1).$

(iv) To show that $\EE^{\vee}(N+m+1)$ is Ulrich, we verify the
conditions of Lemma \ref{lemaH0}. First, $h^0(\EE^{\vee}(N+m))$ is
dual to $h^N(\EE(-N)),$ which is zero by (ii) above. Next, for
$0<i<N,$ $h^i(\EE^{\vee}(n))$ is dual to $h^{N-i}(\EE(m-n))=0,$ so
we see that $\EE^{\vee}$ is an ACM sheaf. Now from
$h^0(\EE^{\vee}(N+m))=0$ it follows that
$h^0(\EE^{\vee}(N+m+1))=h^0(\EE^{\vee}_H(N+m+1)).$ Now $H$ has
dimension $N-1$ and $\omega_H=\OO_H(m+1),$ so by induction
$\EE^{\vee}_H(N-1+m+1+1)=\EE^{\vee}_H(N+m+1)$ is Ulrich, and so
its $h^0$ is $dr.$ Hence also $h^0(\EE^{\vee}(N+m+1))=dr$ and
$\EE^{\vee}(N+m+1)$ is Ulrich.
\end{proof}

\begin{rk}\label{Rmk_even}\rm
If $\EE$ is an Ulrich bundle of rank $r$ on a general hypersurface
$X$ in $\PP^{N+1}$ of dimension $N$ and degree $d\geq 3$, then
$r(d-1)$ must be even. So for example there is no Ulrich bundle of
rank 3 of  a general quartic hypersurface. The reason for this is
that under these hypotheses, because of the theorem of
Noether-Lefschetz, $\Pic X = \ZZ$, so $c_1(\EE) = mH$ for some
integer $m$. Then $\deg \EE = md$. But also $\deg \EE = r(d+g-1)$
by Lemma \ref{lemaHn}, and $X$ being a hypersurface, its linear
curve section is a plane curve of genus $g =
\frac{1}{2}(d-1)(d-2)$. Equating these two expressions for $\deg
\EE$, we find $m = \frac{1}{2}r(d-1)$, so $r(d-1)$ must be even.
\end{rk}

For future reference we compute the Hilbert polynomial of an
Ulrich bundle on a nonsingular variety.
\begin{lema}\label{lemaHilbpoly}
If $\EE$ is an Ulrich bundle of rank $r$ on a nonsingular
projective variety $X$ of degree $d$ and dimension $N$, then its
Hilbert polynomial is $$P_{\EE}(n)=rd{{n+N}\choose{N}}.$$
\end{lema}
\begin{proof}
(See also \cite[2.2]{EisSchreyer}). By induction on $N$, the case
$N=1$ being Proposition \ref{prop1curves} above. For $N \geq 2$,
using the sequence $0 \lra \EE(-1) \lra \EE \lra \EE_H\lra 0$ we
find
$P_{\EE}(n)-P_{\EE}(n-1)=P_{\EE_H}(n)$ which by induction is
$rd{{n+N-1}\choose{N-1}}$.

Since $\EE$ is Ulrich, $H^i(\EE(n))=0$ for all $i>0$ and all
$n\geq 0$. Hence $P_{\EE}(n)=h^0(\EE(n))$ for $n\geq 0$, and
similarly for $\EE_H.$ But $h^0(\EE(n))=\sum_{j=0}^n
h^0(\EE_H(n))=\sum_{j=0}^n rd {{n+N-1}\choose{N-1}}.$ We conclude
by summing the binomial coefficients.
\end{proof}

We now turn our attention to the stability and semistability
property of these bundles.
\begin{defi} Let $X$ be a nonsingular projective variety and let $\EE$ be a vector bundle on it.
Then $\EE$ is said to be \textit{semistable} if for every nonzero
coherent subsheaf $\FF $ of $\EE$ we have the inequality
$$P_{\FF}/\rg(\FF) \leq P_{\EE}/\rg(\EE),$$
where $P_{\FF}$ and $P_{\EE}$ are the Hilbert polynomials of the
sheaves. It is \textit{stable} if one always has strict inequality
above.

The \textit{slope} $\mu(\EE)$ of $\EE$  is defined as
$\deg(c_1(\EE))/\rg(\EE).$ We say that $\EE$ is
$\mu$-\textit{semistable} if for every subsheaf $\FF$ of $\EE$
with $0 < \rg \FF < \rg \EE $, $\mu(\FF) \leq \mu (\EE)$. We say
$\EE$ is $\mu$-\textit{stable} if strict inequality $<$ always
holds. The two definitions are related as follows
$$\mu-stable  \quad \Rightarrow \quad stable \quad \Rightarrow \quad semistable \quad \Rightarrow \quad \mu-semistable$$
(see \cite[1.2.13]{HL}). Note that on a nonsingular curve stable
is equivalent to $\mu$-stable and semistable is equivalent to
$\mu$-semistable.
\end{defi}

\begin{prop}\label{prop2curves}
Let $X,\OO_X(1)$ be a nonsingular projective curve. Any Ulrich
bundle $\EE$ on $X$ is semistable. Furthermore, any coherent
subsheaf $\FF \subset \EE$ with the same slope is also an Ulrich
bundle.
\end{prop}

\begin{proof}
From Proposition \ref{prop1curves} (b) we have $\mu(\EE)=d+g-1$
where $d$ and $g$ are the  degree and genus of $X$ respectively.
If $\FF$ is any coherent subsheaf of $\EE$, then $\FF$ is locally
free and $h^0(\FF(-1))=0$, so by  Proposition \ref{prop1curves}
(b), $\mu(\FF)\leq d+g-1$. Hence $\EE$ is semistable. If $\mu
(\FF) =\mu( \EE)$, then we have equality in Proposition
\ref{prop1curves} (b) for $\FF$, hence $\FF$ is also Ulrich.
\end{proof}

For varieties of higher dimension we have the following result on
Ulrich bundles (if they exist).
\begin{teo}\label{teo_semist}
Let $X, \OO_X(1)$ be a nonsingular projective variety, and let
$\EE$ be an Ulrich bundle on $X$. Then
\begin{itemize}
 \item[(a)] $\EE$ is semistable and $\mu$-semistable
\item[(b)] If $0\lra \FF \lra \EE \lra \GG \lra 0$ is an exact sequence of coherent sheaves with $\GG$ torsion-free,
and $\mu(\FF) =\mu(\EE)$, then $\FF$ and $\GG$ are both  Ulrich
bundles.
\item [(c)] If $\EE$ is stable, then it is also $\mu$-stable.
\end{itemize}
\end{teo}

\begin{proof}
We will use induction on the dimension of $X$.

First suppose $\dim{X}=1$. Part (a) follows by \ref{prop2curves}
since on a curve semistable is equivalent to $\mu$-semistable. For
part (b) we have shown in Proposition \ref{prop2curves} that $\FF$
is Ulrich. Since $\GG$ is assumed to be torsion-free, it is
locally free on the curve $X$. By reason of degrees in the exact
sequence, the slope of $\GG$ will also be $d+g-1$, so by
Proposition \ref{prop1curves}(b) equality, $\GG$ will be Ulrich
also. Part (c) follows directly from the equivalence of stability
and $\mu$-stability on curves.

Now let $\dim X \geq 2$. Given $\EE$, we take a generic hyperplane
section $H$ of $X$, which will also be a nonsingular variety of
the same degree, and consider the restriction $\EE_H$ of $\EE$ to
$H$. By Lemma \ref{lemaHn}, $\EE_H$ is an Ulrich bundle on $H$.
Thus we can apply the induction hypothesis to $\EE_H.$

First we show that $\EE$ is $\mu$-semistable. Indeed, if $\FF
\subset \EE$ is any coherent subsheaf, then $\FF_H$ is a subsheaf
of $\EE_H$ for a general hyperplane section $H$, so
$\mu(\FF_H)\leq \mu(\EE_H)$ since $\EE_H$ is $\mu$-semistable by
induction. But the slope is preserved by passing to a hyperplane
section, so $\mu(\FF)\leq \mu(\EE)$ and  $\EE$ is
$\mu$-semistable.

Next we prove part (b). Let $$0\lra \FF \lra \EE \lra \GG \lra 0$$
be an exact sequence with $\GG$ torsion-free and
$\mu(\FF)=\mu(\EE).$ Choosing $H$ general, we have an exact
sequence
$$0\lra \FF_H \lra \EE_H \lra \GG_H \lra
0$$ with $\GG_H$ also torsion-free, so by the induction
hypothesis, $\FF_H$ and $\GG_H$ are both Ulrich sheaves on $H$.
Using the sequence $0\lra \FF(-1) \lra \FF\lra \FF_H\lra 0$ and
ditto for $\GG$ we conclude that $h^0(\FF(n))=h^0(\GG(n))=0$ for
$n<0.$ Hence $h^0(\FF)\leq ds$, where $s=\rg (\FF)$, and $h^0(\GG)
\leq dt$, where $t=\rg (\GG).$ But the rank of $\EE$ is $r=s+t$,
and $h^0(\EE)\leq h^0(\FF)+h^0(\GG)$, so we must have equality in
both cases. From this we conclude that $H^0(\FF) \lra H^0(\FF_H)$
is surjective and hence $H^0(\FF(n)) \lra H^0(\FF_H(n))$ is also
surjective for all $n\in \ZZ$, since $h^0(\FF_H(-1))=0$ and
$\FF_H$ is generated by global sections. The same holds also for
$\GG$. Therefore $H^1(\FF(n-1)) \lra H^1(\FF(n))$ is injective for
all $n\in \ZZ$, and ditto for $\GG$. But these groups are zero for
$n>>0$ by Serre vanishing, so they are all zero.

For $2\leq i < \dim{X}$, since $\FF_H$ is ACM, we have
$H^{i-1}(\FF_H(n))=0$ for all $n$. Hence $H^i(\FF(n-1))\lra H^i
(\FF(n))$ is injective for all $n$, and again by Serre vanishing
we conclude these groups are all zero.

To conclude that $\FF$ and $\GG$ are both Ulrich sheaves, it
remains to show that they are locally Cohen-Macaulay sheaves,
which amounts on the nonsingular variety $X$, to showing that they
are locally free. The argument is the same for $\GG$ and for
$\FF$, so we write $\FF$ only. To say $\FF$ is locally free is
equivalent to saying $\depth \FF=\dim{X}$, and this is equivalent
to the vanishing of the sheaf $\EE xt^i(\FF,\omega_X)$ for all
$i>0$, since $\omega_X$ is locally free on $X$. Using Serre's
vanishing and the spectral sequence of local and global Ext, this
is equivalent to saying $\Ext^i(\FF,\omega_X(n))=0$ for $i>0$ and
for all $n>>0$. By Serre duality on $X$, this is equivalent to
saying $H^i(X, \FF(-n))=0$ for $i<\dim{X}$ and all $n >>0.$ But
this we have already established, so $\FF$ is locally free, and
hence is an Ulrich sheaf. This completes the proof of part (b) of
the Theorem.

Next, to show that $\EE$ is also semistable, as in  part (a), let
$\FF$ be any coherent subsheaf of $\EE$. If $\mu(\FF)<\mu(\EE)$
then clearly $\frac{P_{\FF}}{\rg\FF} < \frac{P_{\EE}}{\rg\EE},$
because the slope dominates the Hilbert polynomial. On the other
hand, if $\mu(\FF)=\mu(\EE)$, then $P_{\FF} \leq P_{\FF'}$ where
$\FF'$ is a slightly larger sheaf, obtained by pulling back
torsion from $\EE/\FF$, and then by part (b), $\FF'$ is also
Ulrich of the same slope. In this case $\frac{P_{\FF}(n)}{\rg\FF}
\leq \frac{P_{\FF'}(n)}{\rg\FF} =
\frac{P_{\EE}(n)}{\rg\EE}=d{{n+N}\choose{N}}$ where $N=\dim X$ by
Lemma \ref{lemaHilbpoly}. Hence $\EE$ is also semistable.

Finally, to prove part (c), suppose that $\EE$ is stable. We wish
to show that $\EE$ is also $\mu$-stable, i.e. for any $\FF\subset
\EE $, $\mu(\FF)<\mu(\EE)$. We know in any case $\mu(\FF)\leq
\mu(\EE)$. If equality holds, then by (b), replacing $\FF$ by
$\FF'$ as above if necessary, $\FF$ will also be Ulrich, and
$\frac{P_{\FF}}{\rg\FF}=\frac{P_{\EE}}{\rg\EE}$, contradicting the
hypothesis $\EE$ stable.
\end{proof}

In this paper we will use two kinds of moduli spaces. One is the
usual moduli of semistable sheaves, as explained in \cite[Chapter 4]{HL}. 
For fixed rank and Chern classes, we obtain a quasiprojective
variety $M^{ss}_X(r; c_1, \dots,c_r)$ containing an open subset
$M^s_X(r; c_1, \dots,c_r)$ of stable bundles. The points of $M^s$
correspond to isomorphism classes of stable bundles, while the
points of $M^{ss}\setminus M^s$ correspond to $S$-equivalence
classes of semistable bundles (see \cite[4.3.4]{HL}). Here two
semistable bundles are called $S$-\textit{equivalent} if the
associated direct sum of stable bundles that occur as factors in a
Jordan-Hölder sequence are the same. Thus we may regard
$M^{ss}\setminus M^s$ as parametrizing \textit{polystable}
bundles, i.e. direct sums of two or more stable bundles. The space
$M^{ss}$ also corepresents the functor of families of semistable
bundles in the sense that, for any flat family $\EE$ on $X\times T
/ T,$ there is a corresponding morphism of $T \ra M^{ss},$ and
$M^{ss}$ is universal with this property \cite[2.2.1]{HL}.

The property of being Ulrich in a family of vector bundles is an
open condition. Indeed, the property of being locally free is
open, the property of being an ACM sheaf is open, and the property
$H^0(\EE(-1))=0$ is open. Since $H^i(\EE)=0$ for $i>0$ because
$\EE$ is ACM and by Lemma \ref{lemaHn}, we see that $h^0(\EE)$ is
of constant dimension in the family. Thus as soon as there is an
Ulrich bundle, $h^0(\EE)=\chi(\EE)=rd$ is constant and thus the
subset corresponding to Ulrich bundles is open. Therefore we have
an open subset $M^{ss,U}$ in $M^{ss}$ and a corresponding open
subset $M^{s,U}$ of $M^{s}$ corresponding to semistable and stable
Ulrich bundles, respectively.

For a point $m \in M^s$ corresponding to a stable vector bundle
$\EE$, the usual obstruction theory \cite[4.5.2]{HL} allows us to
compute the Zariski tangent space to $M^s$ as $\Ext^1(\EE,\EE),$
and obstructions as $\Ext^2(\EE,\EE).$ Thus in our case, if we
know there exist stable Ulrich bundles, we can estimate the
dimension of the moduli space in this way. However, this method
does not help us prove the existence of stable vector bundles.

To prove the existence of stable bundles, we introduce another
space, the modular family of simple vector bundles. Recall
\cite[\S 26]{HDefTh} that a \textit{modular family} for a class of
objects, say vector bundles on $X$, is a flat family $\EE$ on $X
\times S /S$, with $S$ a scheme of finite type such that
\begin{itemize}
\item[a)] each isomorphism class of bundles occurs at least once,
and at most finitely many times in the family
\item[b)] For each $s \in S,$ the local ring $\widehat{\OO}_{S,s}$
together with the induced family, pro-represents the local
deformation functor
\item[c)] for any other flat family $\EE'$ on $X\times S'/S'$ of
such bundles, there exists a surjective étale map $S'' \ra S'$ for
some scheme $S''$, and a morphism $S''\ra S$ such that
$\EE'\times_{S'} S'' / X\times S'' \cong \EE \times_{S} S'' /
X\times S''.$
\end{itemize}

Recall that a vector bundle $\EE$ on $X$ is \textit{simple} if
$\End(\EE)=k.$ For simple vector bundles, the local deformation
functor is pro-representable \cite[19.2]{HDefTh}. Then, just as in
the case of vector bundles on curves \cite[28.4]{HDefTh}, we
obtain a modular family. Surely such a modular family exists quite
generally, but we give the proof only in a restricted case since
it is easier.

\begin{prop}\label{Prop_modfam} On a nonsingular projective variety $X$,
any bounded family of simple bundles $\EE$ with given rank  
and Chern classes satisfying $H^2(\EE \otimes \EE^{\vee})=0$ 
has a smooth modular family.
\end{prop}

\begin{proof} We follow the proof of \cite[28.4]{HDefTh}. To eliminate the automorphisms induced by scalar multiplication,
we consider pairs $(\EE, \theta)$ as in \cite[28.4]{HDefTh}. We
pick $m>>0$ so that the sheaves $\EE(m)$ are generated by global
sections, and consider the $Quot$ scheme of quotients $\OO^N\lra
\EE(m)\lra0.$  Then the main point is to observe from the sequence
$$0\lra Q \lra \OO^N \lra \EE(m) \lra 0$$ taking its dual and
tensoring with $\EE(m)$, we have
$$0\lra \EE \otimes \EE^{\vee} \lra \EE(m)^N \lra Q^{\vee} \otimes\EE(m) \lra
0.$$ Since $H^1(\EE(m))=0$ for $m>>0$, and $H^2(\EE\otimes
\EE^{\vee})=0$, we find that $H^1(Q^{\vee}\otimes \EE(m))=0,$ so
the Quot scheme of quotients $\OO^N \lra \EE(m)\lra 0$ is smooth,
and also $H^0(Q^{\vee} \otimes\EE(m))\ra \Ext^1(\EE,\EE)$ is
surjective. This enables us to show, as in the proof of
\cite[27.2]{HDefTh}, that $\widehat{\OO}_{S,s}$ pro-represents the
local deformation functor. The rest of the proof proceeds as in
\cite[28.4]{HDefTh}.
\end{proof}

\begin{rk}\label{Rem_modfam}\rm
In particular, there is a smooth modular family for simple Ulrich
bundles on the cubic surface of given rank and first Chern class.
Indeed, they form a bounded family by Proposition
\ref{prop_irredfam} below, $c_2$ is determined from $c_1$ and $r$
(because the Hilbert polynomial of an Ulrich bundle is determined
by its rank by \ref{lemaHilbpoly}) and $H^2(\EE \otimes
\EE^{\vee}) \perp H^0((\EE \otimes \EE^{\vee})(-1))$ which is 0
because $\EE$ is simple.
\end{rk}

To aid in the computation of dimensions of families of extensions
and moduli spaces, we compute $\chi(\EE \otimes \FF^{\vee})$ for
Ulrich bundles $\EE$, $\FF$ on any algebraic surface.

\begin{prop}\label{lemachi} Let $X$ be a nonsingular projective surface
of degree $d$. Let $\EE$, $\FF$ be Ulrich bundles of ranks $r$,
$s$, respectively, and with first Chern classes $C,D$,
respectively. Then
$$\chi(\EE \otimes \FF^{\vee})= rD.K-C.D+rs(2d-1-p_a),$$
$$\chi(\EE \otimes \EE^{\vee})= rC.K-C^2+r^2(2d-1-p_a),$$
where $K$ is the canonical divisor on $X$ and $p_a$ is the
arithmetic genus of $X$.
\end{prop}

\begin{proof}
Let $c_2$ be the second Chern class of $\EE$. Then Riemann-Roch
theorem says
\begin{equation}\label{eq_chi}\chi(\EE)=\frac{1}{2}(C^2-2c_2)-\frac{1}{2}C.K
+r(1+p_a).
\end{equation}
Now call $d_2$ the second Chern class of $\FF.$ A straightforward
computation of Chern classes shows that $c_1(\EE \otimes
\FF^{\vee})=sC-rD$ and
$$c_2(\EE \otimes \FF^{\vee})=\frac{1}{2}s(s-1)C^2+sc_2+\frac{1}{2}r(r-1)D^2+rd_2-(rs-1)C.D.$$
Applying Riemann-Roch to the bundle $\EE \otimes \FF^{\vee}$ of
rank $rs,$ we find
$$\chi(\EE \otimes \FF^{\vee})=\frac{1}{2}(sC^2+rD^2)-sc_2-rd_2-C.D-\frac{1}{2}(sC-rD).K+rs(1+p_a).$$
Now as $\EE$ is an Ulrich bundle of rank $r$ on the surface $X$ of
degree $d$ we know that $\chi(\EE)=rd$ (see Lemma
\ref{lemaHilbpoly}). But by Riemann-Roch this is also equal to
(\ref{eq_chi}), so we can solve for $c_2$ and substitute in the
above formula for $\chi(\EE \otimes \FF^{\vee}).$ Doing the same
for the second Chern class $d_2$ of $\FF$ we obtain the desired
result.
\end{proof}

\begin{cor}\label{corchi}
Let $X$ be a del Pezzo surface of degree $d$. Let $\EE$, $\FF$ be
as in Proposition \ref{lemachi}. Then $\chi(\EE \otimes
\FF^{\vee})= (d-1)rs-C.D$ and $\chi(\EE \otimes
\EE^{\vee})=(d-1)r^2-C^2.$
\end{cor}

\begin{proof}
This immediately follows from Proposition \ref{lemachi} because in
this case $p_a=0$, $K=-H$, $\deg C=rd$ and $\deg D=sd$ by Lemma
\ref{lemaHn}.
\end{proof}

\begin{cor}\label{cordim}
If $\EE$ is a simple or stable Ulrich bundle of rank $r$ and
$c_1(\EE)=D$ on the cubic surface $X$, then the modular family or
the coarse moduli space at that point is smooth of dimension
$D^2-2r^2+1.$
\end{cor}

\begin{proof}
The modular family or the moduli space is smooth because $H^2(\EE
\otimes \EE^{\vee})=0$ (cf. Remark \ref{Rem_modfam}). Then, as
$h^0(\EE \otimes \EE^{\vee})=1$ for simple and stable bundles, the
dimension of this space is $h^1(\EE \otimes \EE^{\vee})=1-\chi(\EE
\otimes \EE^{\vee})$, which by Corollary \ref{corchi} is equal to
$D^2-2r^2+1.$
\end{proof}

\section{Existence of Ulrich sheaves on cubic surfaces}

In this section for any Ulrich bundle $\EE$ on the nonsingular
cubic surface $X$, we denote by $D \in \Pic X$ the divisor class
of its first Chern class $c_1(\EE)$. We express the Ulrich bundle
as an extension of the twisted ideal sheaf $\JJ_Z(D)$ for a
certain zero-scheme $Z$ by a trivial sheaf $\OO_X^{r-1}$. This
allows us to show that the Ulrich bundles of given rank and first
Chern class $D$ form an irreducible family in the moduli space.
For ranks 1 and 2 we list the possible divisors $D$ that
correspond to Ulrich bundles, using ad hoc arguments. The
technical heart of this paper is Proposition \ref{PropAlg} that
shows for a divisor $D$ satisfying certain incidence conditions,
how to choose a twisted cubic curve $T$, so that the new divisor
$D' = D - T$ satisfies the same conditions. To prove this we
express $D$ in a ``standard form" in $\Pic X$, and do a
case-by-case analysis. This proposition is used in the induction
needed to characterize those divisors $D$ corresponding to Ulrich
bundles of any rank, in terms of their intersection numbers with
lines. The same proposition is used again in the next section to
characterize those $D$ that correspond to stable Ulrich bundles.

\begin{prop}\label{prop_irredfam}

\begin{enumerate}
 \item[(a)] Let $\EE$ be an Ulrich bundle of rank $r\geq 2$ on the cubic surface $X$. Then there is an exact sequence
\begin{equation*}
0 \lra \OO_X^{r-1} \lra \EE \lra \JJ_Z(D) \lra 0
\end{equation*}
where $D$ is a divisor of positive degree, $Z$ is a zero-scheme in
$X$, and furthermore $H^0(\JJ_Z(D-H))=0$.
\item[(b)] Conversely, given an integer $r\geq 2$, a divisor $D$ of positive degree and a zero-scheme $Z$ of length $n$
such that $h^0(\JJ_Z(D-H))=0$, the collection of coherent sheaves
$\EE$ obtained as extensions
as above forms an irreducible family (for $r$, $D$, $n$ fixed and
$Z$ and choice of extension variable).
\end{enumerate}
\end{prop}

\begin{proof}
(a) Since $\EE$ is generated by global sections by definition, a
well-known lemma states that dividing by $r-1$ general sections
will leave a quotient of rank 1 that is torsion-free. Any such
torsion-free sheaf can be written as $\JJ_Z(D)$ for some divisor
$D$ and some zero-scheme $Z$. Since $D$ represents $c_1(\EE)$, we
have $\deg D=3r>0$.

Since $\EE$ is Ulrich, we have $H^0(\EE(-1))=0$. Since also
$H^1(\OO_X(-1))=0$ , we find $H^0(\JJ_Z(D-H))=0$.

(b) We first note that the Hilbert scheme of zero-schemes in $X$
of length $n$ is irreducible \cite{Fogarty}. The condition
$H^0(\JJ_Z(D-H))=0$ is an open condition on $Z$, so the family of
$Z's$ used here is irreducible. Next we observe that such an
extension is defined by a choice of $r-1$ elements of
$\Ext^1(\JJ_Z(D),\OO_X)$ which is dual to $H^1(\JJ_Z(D-H)).$
Because of the hypothesis $H^0(\JJ_Z(D-H))=0$ there is an exact
sequence
$$0 \ra H^0(\OO_X(D-H))\ra H^0(\OO_Z) \ra H^1(\JJ_Z(D-H)) \ra H^1(\OO_X(D-H))\ra 0.$$
Now $H^2(\OO_X(D-H))$ is dual to $H^0(\OO_X(-D))$, which is zero
because $D$ has positive degree. Therefore
$$h^1(\JJ_Z(D-H))=\length Z - \chi(\OO_X(D-H))$$
depends only on $D$ and $n$. Thus the dimension of
$\Ext^1(\JJ_Z(D),\OO_X)$ is constant as $Z$ varies, so the family
of all these sheaves is parametrized by a Grassmannian of a vector
bundle over an open subset of the Hilbert scheme, and hence is an
irreducible family.
\end{proof}

\begin{cor}\label{cor_irred}
The Ulrich bundles $\EE$ on the cubic surface $X$ of given rank
and first Chern class $c_1(\EE)\in \Pic(X)$ (if they exist) form
an irreducible family. More precisely, there is a smooth
irreducible variety $T$ and a vector bundle $E$ on $X\times T$
such that each fiber for $t\in T$ is an Ulrich bundle on $X$ with
given rank and $c_1(\EE)$, and all such Ulrich bundles appear in
the family $E$ (in general many times) .\end{cor}

\begin{proof}
Since $\EE$ is an Ulrich bundle, its Hilbert polynomial is already
determined by its rank (see Lemma \ref{lemaHilbpoly}), so in
particular the length $n$ of the zero-scheme $Z$ of part (a) of
the previous proposition, which represents the second Chern class
of $\EE$, is determined by $r$. Then using the given $r,D,n$, we
consider the irreducible family of coherent sheaves described in
\ref{prop_irredfam}(b). The property of being Ulrich is open (see
discussion on moduli spaces in section 2).  This open subset gives
the parameter space $T$ for the required irreducible family.
\end{proof}


\begin{cor}\label{cor_modspace} The moduli space $\MM^{ss,U}_X(r;D,n)$ of semistable Ulrich bundles  is irreducible.
\end{cor}
\begin{proof}
As all Ulrich bundles are semistable by Theorem \ref{teo_semist},
there is a map from the  irreducible family in Corollary
\ref{cor_irred} to the coarse moduli space $\MM^{ss}_X(r;D,n)$.
The image of this map is $\MM^{ss,U}_X(r;D,n)$ which is therefore
irreducible.
\end{proof}

\begin{prop}\label{cor_nonsing}
If $\EE$ is an Ulrich bundle of rank $r\geq 1$ on a cubic surface
$X$ with first Chern class $D$, then
\begin{itemize}
\item[a)] $\deg D=3r$
\item[b)] $D^2=2c_2(\EE)+r>0$
\item[c)] $0\leq D.L \leq 2r$ for all lines $L$ on $X$
\item[d)] $D$ can be represented by an irreducible nonsingular curve.
\end{itemize}
\end{prop}

\begin{proof} a) $\deg D=\deg \EE =3r$ by Lemma
\ref{lemaHn}.

b) The  Riemann-Roch theorem applied to $\EE$ says
$$\chi(\EE)=r+\frac{1}{2}c_1(\EE).H+\frac{1}{2}(c_1(\EE)^2-2c_2(\EE)).$$
On the other hand, by Lemma \ref{lemaHilbpoly} $\chi(\EE)=3r$.
Since $c_1(\EE)=D$ and $D.H=3r$ we can solve for $c_1(\EE)^2=D^2$
to get b). Note that by Proposition \ref{prop_irredfam}, the
second Chern class $c_2(\EE)$ is represented by an effective
zero-scheme, so $c_2(\EE)\geq 0$.

c) If $\EE$ is an Ulrich bundle of rank $r$  with $c_1(\EE)=D$,
then it has a resolution over $\PP^3$ of the form
$$0\lra \OO_{\PP^3}(-1)^{3r}\lra \OO_{\PP^3}^{3r} \lra \EE \lra 0 $$
by \cite[3.7]{CH2}. If $L$ is a line in $X$, restricting to $L$ we
get
$$\dots \lra \OO_{L}(-1)^{3r}\srel{\alpha}{\lra} \OO_{L}^{3r} \lra \EE_L \lra 0.$$
The image of $\alpha$ must be a subbundle of $\OO_L^{3r}$ of rank
$2r$, which is also a quotient of $\OO_L(-1)^{3r}.$ On $L$, vector
bundles are completely decomposable, so it must be of type
$\OO_L(-1)^s\oplus \OO_L^t$ for some $s,t\geq 0$, $s+t=2r.$ The
degree of this bundle is $-s$, so the degree of the quotient
$\EE_L=\OO_L^{3r}/\im \alpha$ (which is equal to $D.L$) is $s$ and
$0\leq s \leq 2r.$

d) This follows from b) and c) \cite[V, Exercise 4.8]{AG}.
\end{proof}

\begin{ex}\label{exrk1}
\rm \textbf{Ulrich bundles of rank 1}. These will be  line bundles
$\LL$, with $\deg \LL=3$, $h^0(\LL)=3$, $h^0(\LL(-1))=0$, and
$\LL$ generated by global sections. Thus $\LL$ corresponds to a
curve of degree $3$ irreducible and nonsingular by Proposition
\ref{cor_nonsing}. The only such curves on the cubic surface are
$H$, (a hyperplane section, in which case $h^0(\LL(-1))\neq 0$)
and twisted cubic curves. Thus $\LL \cong \OO(T)$ where $T$ is a
twisted cubic curve. There are 72 classes of twisted cubic curves
in $\Pic X$, represented by (using notation of \cite[V, \S 4]{AG})
$$\begin{array}{rcl}
T_A & = & (1; 0,0,0,0,0,0)\\
T_B & = & (2;1,1,1,0,0,0) \\
T_C & = & (3;2,1,1,1,1,0) \\
T_D & = & (4;2,2,2,1,1,1) \\
T_E & = & (5;2,2,2,2,2,2) \\
\end{array}$$
and their permutations. The permutations give 1 of type $T_A$, 20
of $T_B$, 30 of $T_C$, 20 of $T_D$ and 1 of $T_E$, making 72 in
all. For future reference we record that $T^2=1$ for any twisted
cubic curve $T$, and that the intersection of two distinct classes
can be 2,3,4 or 5 (for example $T_A.T_B=2$, $T_A.T_C=3$,
$T_A.T_D=4$, $T_A.T_E=5$).
\end{ex}

\begin{ex}\rm\label{exrk2}\textbf{Rank 2 Ulrich bundles on the cubic surface}.
We start by listing the possible irreducible nonsingular curves
$D$ (see Proposition \ref{cor_nonsing}) of degree 6 on the cubic
surface, given as divisors $(a;b_1, \dots,b_2)$ in standard form
with $a\geq b_1+b_2+b_3$, and $b_1 \geq \dots \geq b_6$ (see
Proposition \ref{Propstdform}).

\begin{center}
\begin{tabular}{|c|c|c|c|c|c|c|}
  & $D$&$D^2$&$ \sum T_i$&$\exists\textrm{  Ulrich}$&$\dim\{\textrm{ss simple}\}$&$\dim\{$stable$\}$\\
  \hline
$a$&(2;0,0,0,0,0,0)&4&$2A$&$\checkmark$& -&-\\
$b$&(3;1,1,1,0,0,0)&6&$A+B$&$\checkmark$&-&-\\
$c$&(3;2,1,0,0,0,0)&4&-& &-&-\\
$d$&(4;2,1,1,1,1,0)&8&$A+C$&$\checkmark$&0&1\\
$e$&(4;1,1,1,1,1,1)&10&$B+B'$&$\checkmark$&1&3\\
$f$&(6;2,2,2,2,2,2)&12&$A+E$&$\checkmark$&2&5\\
\hline
\end{tabular}
\end{center}
We have written the various divisors $D$ as sums of twisted cubic
curves (labeled as in Example \ref{exrk1}) in the third column, by
inspection. The sign $'$ on a letter, such as $B'$, means a
permutation of the form listed for that letter. In this case,
$B'=(2;0,0,0,1,1,1)$. The column $\dim\{\textrm{ss simple}\}$
refers to the dimension of properly semistable simple Ulrich
bundles and $\dim\{\textrm{stable}\}$ to the dimension of stable
Ulrich bundles. We will justify these two last columns later in
section 4.

To determine the possible existence of an Ulrich bundle with
$c_1=D,$ first note that if there are stable bundles corresponding
to $D,$ then the dimension of the moduli space is
$D^2-2r^2+1=D^2-7$ (see Corollary \ref{cordim}), and this number
must be nonnegative, so $D^2\geq 7$.  Thus for the first three
cases $a,b,c$ above there can be no stable Ulrich bundles. If
there is any Ulrich bundle $\EE$ at all, then it must be
semistable by Theorem \ref{teo_semist}, hence an extension of two
rank 1 Ulrich bundles: take any rank 1 subsheaf $\FF$ of the same
slope, pull back torsion of $\GG=\EE/\FF$ if necessary, and apply
Theorem \ref{teo_semist}. Each of these is a twisted cubic curve,
so $D=T_1+T_2$ is a sum of two twisted cubic curves. Since the
divisor in case $c$ above cannot be written in this way, there is
no Ulrich bundle corresponding to that $D$. On the other hand, in
the other five cases, just taking a direct sum $\OO(T_1) \oplus
\OO(T_2)$ shows the existence of Ulrich bundles for those values
of $D$.
\end{ex}

\begin{prop}\label{Propstdform}{\rm (Standard form)}
Let $D$ be a divisor on a nonsingular cubic surface $X$.
\begin{itemize}
\item [(a)] There is a choice of 6 points in $\PP^2$ such that when $X$ is represented as the blow-up of
these 6 points, $D$  can be represented as $(a; b_1, \dots,b_6)$
with $b_1 \geq \dots \geq b_6$  and $a\geq b_1+b_2+b_3$ in the
notation of \cite[V, \S 4]{AG}.
\item [(b)] The integers  $a,b_1, \dots,b_6$ in a representation satisfying
the conditions of \textrm{(a)} are uniquely determined. In
particular, $b_3,b_4,b_5,b_6$ are the four smallest numbers of the
set $\{D.L \mid L \textrm{ is a line on } X\}.$
\item [(c)] In that representation $b_6=\min\{D.L \mid L \textrm{ is a line on } X\}$ and
$D.G_1=2a-\sum_{i=2}^6 b_i=\max \{D.L \mid L \textrm{ is a line on
} X\}$ ($G_1$ refers to the notation of \cite[V, 4.9]{AG}). One
can also write $D.G_1=d-a+b_1$ where $d= \deg D.$
\item [(d)] The integer $a$ can be characterized as
$$a =\min \{D.T \mid T \textrm{ is a twisted cubic on }X\}.$$
\end{itemize}
\end{prop}

\begin{proof}

(a) (cf. \cite[proof of V, 4.11]{AG}) Given $D$ choose six
mutually skew lines $E_1', \dots, E_6'$ as follows. Take $E_6'$ so
that $D.E_6'=\min\{D.L \mid L \textrm{ line on } X\}$. Choose
$E_5'$ such that $D.E_5'$ is minimum among lines that do not meet
$E_6'$. Choose $E_4'$ and $E_3'$ similarly such that $D-E_i'$ is
minimum among those lines not meeting any of the $E_j'$ already
chosen. There remain three lines not meeting any of
$E_3',E_4',E_5',E_6',$ one of which meets the other two. Take the
two that do not meet each other to be $E_1'$, $E_2'$ with
$D.E_1'\geq D.E_2'.$ Then $E_1',\dots E_6'$ being six mutually
skew lines, we can represent them as $E_1, \dots E_6$ for a
suitable projection $X \ra \PP^2$ \cite[V, 4.10]{AG}. We write $D=
al-\sum b_ie_i$ in this basis for statement (a). As $b_i=D.E_i$,
by the choices we have made we  have $b_1 \geq b_2 \geq \dots \geq
b_6$.  Note that $F_{12}$ was available (notation of \cite[V,
4.9]{AG}), not meeting $E_4,E_5,E_6$ at the time we chose $E_3$.
Therefore $D.F_{12}\geq D.E_3$, i.e $a-b_1-b_2 \geq b_3$. This
says $a\geq b_1+b_2+b_3$.

(b) Suppose a divisor $D$ is represented as $(a;b_1, \dots,b_6)$
as in (a) with $b_1 \geq \dots \geq b_6$ and $a \geq b_1+b_2+b_3$.
Then (using the notation $E_i$, $F_{ij}$, $G_i$ of \cite[V
4.9]{AG}) we have
$$
\left\{\begin{array}{lcl}
D.E_i & = & b_i\\
D.F_{ij}& =& a-b_i-b_j\\
D.G_j & = & 2a-\sum_{i\neq j} b_i
\end{array}
\right.
$$
From the condition $a \geq b_1+b_2+b_3$ it follows that the
minimum of $D.F_{ij}$ is $D.F_{12} \geq b_3.$ The maximum of the
$D.F_{ij}$ is $D.F_{56}=a-b_5-b_6.$ This is $\leq D.G_4=
2a-\sum_{i\neq 4}b_i$ since $a \geq b_1+b_2+b_3.$ Clearly $D.G_4
\leq D.G_3 \leq D.G_2 \leq D.G_1.$ Moreover $D.G_5 \geq D.G_6$
which is $\geq b_3$ because $a \geq b_1+b_2+b_3 \geq b_3+b_4+b_5$.
Thus the recipe of (a) will chose the same $E_i$, and we have the
same representation. Note that even though we chose $E_5'$ minimum
among the lines not meeting $E_6'$, and similarly $E_4', E_3'$, it
follows from the proof of (b) that $b_3, b_4, b_5, b_6$ are the
four smallest values of $\{D.L\}.$

(c) From the analysis just given,  $b_6=\min \{D.L\}$ and $D.G_1=
\max\{D.L\}.$ Also $D.G_1=2a- \sum_{i\neq 1}b_i=d-a+b_1$ since
$d=3a-\sum_{i=1}^6 b_i.$

(d) To compute $D.T$ for various twisted cubic curves $T$, we use
the notation in Example \ref{exrk1} above and their permutations.
Note that the basis for $\Pic X$ was chosen to make $D$ have
standard form, so the $T's$ need not be in standard form. Then
\begin{itemize}
 \item $D.T_A=a$
\item $D.T_B=2a-b_1-b_2-b_3 \geq a.$ If we use a permutation of $B$ we get $2a-b_i-b_j-b_k \geq a$ because $b_1 \geq b_2 \geq \dots  \geq b_6.$
\item $D.T_C=3a-2b_1-b_2-b_3-b_4-b_5$. Again, since $a\geq b_1+b_2+b_3$ and $a\geq b_1+b_4+b_5$ this is $\geq a$.
Moreover, this is the minimum of $D.T_C$ for any permutation of
$C$.
 \item $D.T_D=4a-2(b_1+b_2+b_3)-(b_4+b_5+b_6).$ Again, even after permutation of $T_D$ this is $\geq a$.
\item $D.T_E=5a-2\sum b_i \geq a.$
\end{itemize}
Thus $a=D.T_A=\min\{D.T\mid \textrm{ T twisted cubic curve}\},$
which proves (d).
\end{proof}

Next we describe an algorithm that will be used later in some
inductive proofs. Given a divisor $D$ of degree $3r$ we subtract a
suitable twisted cubic curve to get a new divisor $D'$ of degree
$3r-3$.

\begin{prop}\label{PropAlg} Let $D$ be a divisor on $X$ of degree  $3r$,
with $r\geq 2$, satisfying $0\leq D.L \leq 2r$ for all lines $L$.
Write $D$ as $(a; b_1, \dots,b_6)$ in standard form (see
Proposition \ref{Propstdform}). Then
\begin{enumerate}
 \item [(a)] There exists a twisted cubic curve $T$ with $D.T=a$ so that the
 new divisor $D'=D-T$ verifies  $0\leq D'.L \leq 2r-2$ for all lines $L$.
\item[(b)] Furthermore, assuming $r\geq 3$, if $a \geq 2r$ and $D $ is not a
multiple of $D_0=(4;2,1^4,0)$, then $D'$ is not a multiple $mD_0$
for any $m\geq 2$, and also $a'\geq 2r-2 $ where
$D'=(a';b_1',\dots,b_6')$ is in standard form.
\end{enumerate}
\end{prop}

\begin{proof}
We distinguish the following cases.

\underline{Case 1:} $a >b_1+b_2+b_3.$ Take $T=T_A=(1;0^6)$. Then
$D.T=a$ and $D'=D-T=(a-1;b_1, \dots,b_6).$ By hypothesis $a-1 \geq
b_1+b_2+b_3$,  so it is still standard form. Since $D.L\geq 0$ for
all $L$, we have $b_6 \geq 0$, so $D'.L\geq 0$ for all $L$ also.
To find the maximum of $D'.L$, we use Proposition
\ref{Propstdform}(c) to write it as $d'-a'+b_1'$ where  $d'=\deg
D'$. Now $D.L \leq 2r$ for all $L$, so $d-a+b_1\leq 2r.$ But
$d'=3r-3$, $a'=a-1$ and $b_1'=b_1$, so $d'-a'+b_1' \leq 2r-2$.

If $a \geq 2r$, then $a'=a-1 \geq 2r-1 \geq 2r-2$ as required.
Moreover if $a \geq 2r$ then $D'$ cannot be a multiple of $D_0$.
Indeed, if it were, say $D'=mD_0$ with $m\geq 2$, then $a-1=4m$,
and $\deg D'=3r-3=m\deg D_0=6m$, so $r=2m+1.$ But $a=4m+1 \geq
2r,$ a contradiction.

\underline{Case 2:} $a=b_1+b_2+b_3$ and either $b_2 >b_4$ or
$b_3>b_5.$ We take $T=T_B$, so $D.T=2a-\sum_1^3 b_i$. Since
$a=b_1+b_2+b_3$ this gives $D.T=a$. Write $D'=D-T_B=(a-2;
b_1-1,b_2-1,b_3-1,b_4,b_5,b_6).$ Then we reorder the six latter
numbers so that they are in descending order.

Case 2a. If $b_2>b_4$, then $b_4$ may move up to the third place
(in case $b_3=b_4$), but no further. So the first three $b_i'$
include $b_1-1,b_2-1$ and either $b_3-1$ or $b_4$. In either case
$a-2$ is $\geq$ their sum.

Case 2b. If $b_3>b_5$, then again $b_4$ may move up to the first
place (if $b_1=b_2=b_3=b_4$), but we still  have $b_1-1$, $b_2-1$
in the first three positions so again $a-2 \geq$ their sum.

Hence $D'$ is in standard form and since $b_6 \geq 0$, $D'.L \geq
0$ for all $L$. Moreover, $d-a+b_1 \leq 2r$ implies $d'-a'-b_1'
\leq 2r-2$ unless $3r-a+b_1=2r$ and $b_1=b_2=b_3=b_4=:b.$ In this
case $a=r+b.$ But $a=b_1+b_2+b_3=3b$, so $r=2b$. However this
contradicts the fact that degree of $D$ is equal to $3r=6b$ as the
degree is given by $3a-\sum b_i= 5b-b_5-b_6.$

Moreover, assume now that $a \geq 2r.$ Then $a'=a-2 \geq 2r-2$.
Furthermore $D'$ cannot be equal to $mD_0$ for any $m\geq 2$,
because then $D=(4m+2;2m+1,(m+1)^2,m^2,0),$ which is not in
standard form.

\underline{Case 3:} $a=b_1+b_2+b_3$, $b_2=b_3=b_4=b_5 >b_6$. Take
$T=T_C$ so that $D.T_C=3a-2b_1-\sum_2^5b_i=a$. Then
$D'=D-T_C=(a-3; b_1-2,b_2-1,b_3-1,b_4-1,b_5-1,b_6).$
 We first verify that $b_1 \geq 2$.
If not, then $D=(a;1^5,0)$ and $\deg D=3a-5$ which is not a
multiple of 3. Hence $b_1 \geq 2$ so all the new $b_i'$ are $\geq
0$. Rearrange $b_i$ in descending order: since $b_5>b_6$ we still
have $b_6$ at the last place. Hence the first three places will
contain $b_2-1,b_3-1,b_4-1$ or possibly $b_1-2$ in place of one of
these if $b_1>b_2+1$, so $a-3\geq b_1'+b_2'+b_3'$. Thus $D'$ is in
standard form, so $b_6'=b_6 \geq 0$ and hence $D'.L \geq 0$ for
any line $L$.

If $b_1 >b_2$ then $d'-a'+b_1'=3r-3-(a-3)+b_1-2$ so the condition
$D.L \leq 2r$ for all lines implies $D'.L \leq 2r-2$ for all
lines. If $b_1=b_2$ then $d'-a'+b_1'=3r-3-(a-3)+b_1-1$ so  $D'.L
\leq 2r-2$ also for all lines unless max $D.L=2r$. In that case
$b_1=b_2,$ $d-a+b_1=2r,$ and $D=(3b;b,b,b,b,b,b_6).$  Then $d=
2r+2b=3r,$ so $r=2b$ and $d=6b$. This implies $b_6=-2b$, a
contradiction.

Now if $a > 2r$, then $a'=a-3 \geq 2r-2$ and we conclude. If
$a=2r$, then $a'=2r-3$ so we have to prove that this case does not
happen. In this case we obtain $b_1=2r-2b$ for $b=b_2=\dots=b_5$
and $D$ has degree  $ 6r-(2r-2b)-4b-b_6$ which must be equal to
$3r$ so $r= 2b+b_6$. But $d-a+b_1 \leq 2r$, which holds only if $r
\leq 2b$. Therefore $b_6=0$, $r=2b$ and $D=(4b;2b,b^4,0)=bD_0,$
contrary to hypothesis.

Finally, if $D'=mD_0$ with $m\geq 2$, then $D=
(4m+3;2m+2,(m+1)^4,0)$ which cannot hold as $D$ is in standard
form.

\underline{Case 4:} $a=b_1+b_2+b_3$, $b_1>b_2=b_3=b_4=b_5
=b_6=:b.$ Notice that in this case we have $b\neq 0$ because a
divisor of type $(b_1;b_1,0^5)$ has degree $2b_1=3r$ and
$d-a+b_1=2b_1$ which is not $\leq 2r$.

As in case 3 we take $T=T_C$, thus $D.T_C=a$ as before. Then
$D'=D-T_C=(a-3;b_1-2,b-1,b-1,b-1,b-1,b)$. Now $b_6$ may move up,
even to the first place, but we still have $b_1-2$, $b-1$ among
the first three, so $a-3\geq$ the sum of the first three
positions, so $D'$ is in standard form. Since $b_6>0$, $D'.L \geq
0$ for all $L.$ Also $D'.L \leq 2r-2$ as in case 3.

If $D$ satisfies $a > 2r$ then $a'=a-3$ is $\geq 2r-2.$ The case
$a=2r$ cannot occur. Indeed, it implies $b_1=2r-2b$ and degree of
$D$ equal to $4r-3b$ so $r=3b.$ But then $D$ does not satisfy the
hypothesis $d-a+b_1 \leq 2r$ because $d-a+b_1=7b$.

Finally note that in this case $D$ and $D'$ are never multiples of
$D_0$ because $b_6=b_6'>0.$

\underline{Case 5:} $a=3b$ and all $b_i$ are equal to $b\geq 0$.
In this case $\deg D=3\cdot 3b-6b=3b$ so $r=b$ and $D=rH$ where
$H=(3;1^6)$ is the hyperplane class.

We take $T=T_E$ so $D.T_E=5\cdot3b-2\cdot 6\cdot b=3b=a$. Then
$D'=D-T_E=(3b-5; (b-2)^6)$. This is in standard form and it still
satisfies $D'.L\geq 0$ for all $L$. Also $\max
\{D'.L\}=d'-a'+b_1'= b \leq 2r.$ Note that $D$ and $D'$ are never
multiples of $D_0$ in this case. Moreover, assuming $r\geq 3$, as
$a=3r$ we have $a \geq 2r+3$ and so $a'=a-5\geq 2r-2$.
\end{proof}

\begin{teo}\label{condD}
Let $D$ be a divisor on $X$ and let $r \geq 1$ be an integer. Then
the following are equivalent
\begin{itemize}
 \item [(i)] $D$ is linearly equivalent to  a sum of twisted cubic curves $\sum_{i=1}^r T_i$.
\item [(ii)] There exists an Ulrich bundle of rank $r$ with first Chern class equal to $D$.
\end{itemize}
Moreover if $r\geq 2$ these two conditions are equivalent to
\begin{itemize}
\item [(iii)] $\deg D=3r$ and $0 \leq D.L \leq 2r$ for all lines
$L$ in $X$.
\end{itemize}
\end{teo}

\begin{proof}
For $r=1$ we refer to Example \ref{exrk1}. For $r \geq 2$ we will
prove (i) $\Rightarrow$ (ii) $\Rightarrow$ (iii) $\Rightarrow$
(i).

(i) $\Rightarrow$ (ii.) Take the bundle $\EE= \oplus_i \OO(T_i).$

(ii) $\Rightarrow$ (iii). Follows from Proposition
\ref{cor_nonsing} a) and c).

(iii) $\Rightarrow$ (i). We proceed by induction on $r$. If $r=2$,
then $D$ has degree 6 and we just need to look at the table in
Example \ref{exrk2} to  note that any divisor $D$ satisfying
$0\leq D.L \leq 4$ is linearly equivalent to a sum of two twisted
cubic curves (use $\max\{D.L\}=6-a+b_1$). If $r > 2$, by
Proposition \ref{PropAlg} (a) we choose a twisted cubic curve $T$
so that $D'=D-T$ satisfies the same condition (iii). Therefore
$D'$ is itself linearly equivalent to a sum of $r-1$ twisted cubic
curves, so $D=T+D'$ is a sum of $r$ twisted cubic curves.
\end{proof}

\begin{rk}
\rm a) In particular, this theorem implies the existence of
orientable Ulrich bundles of rank $r\geq 2$ on the cubic surface,
that is, Ulrich bundles with $c_1=rH$. Indeed, $rH$ satisfies
condition (iii) of Theorem \ref{condD}. Note that in this proof we
did not need the minimal resolution conjecture for points on X
\cite[4.3]{CH2}, which was an essential ingredient of the
existence proof in our earlier paper \cite{CH2}.

b) Actually, the existence of orientable Ulrich bundles of rank
$r$ on the cubic surface $X$ is equivalent to the existence of
sets of $\frac{1}{2}(3r^2-r)$ points on $X$ satisfying the Minimal
Resolution Conjecture on $X$.

Indeed, it was proven in \cite[4.4(b)]{CH2} that if $Z$ is a set
of $\frac{1}{2}(3r^2-r)$ points on $X$ satisfying the Minimal
Resolution Conjecture,  then there exists an extension
\begin{equation}\label{extension}
0 \lra \OO_X^{r-1} \lra \EE \lra \JJ_Z(r) \lra 0
\end{equation}
where $\EE$ is an orientable Ulrich bundle of rank $r.$

Conversely, given an orientable Ulrich bundle $\EE$ of rank $r$,
by Proposition \ref{prop_irredfam} (a), there is an exact sequence
like (\ref{extension}). As $\EE$ is generated by its global
sections, $\JJ_Z(r)$ is also generated by global sections and
$h^0(\JJ_Z(r))=2r+1.$ On the other hand, $\EE$ has the following
minimal free resolution in $\PP^3$ (see \cite[3.7(b)]{CH2})
$$0 \lra \OO_{\PP^3}(-1)^{3r}\lra \OO_{\PP^3}^{3r}\lra \EE\lra 0.$$
Using the mapping cone procedure with the free resolutions of
$\OO_X$ and $\EE$ in sequence (\ref{extension}) we get the
following free resolution of $\JJ_Z(r)$
$$0\lra \OO_{\PP^3}(-3)^{r-1}\lra \OO_{\PP^3}(-1)^{3r} \oplus \OO_{\PP^3}^{r-1}\lra \OO_{\PP^3}^{3r}\lra \JJ_Z(r)\lra 0.$$
The terms $\OO_{\PP^3}^{r-1}$ can be split off because $\JJ_Z(r)$
is generated by global sections and $h^0(\JJ_Z(r))=2r+1.$
Therefore $\JJ_Z$ has the minimal free resolution predicted by the
Minimal Resolution Conjecture.

c) Using this, we obtain a new proof for the Minimal resolution
conjecture for general sets of $\frac{1}{2}(3r^2-r)$ points on
$X$, which was firstly proven in \cite{Cmrc}.
\end{rk}

\section{Stable Ulrich bundles on the cubic surface}
In this section we prove our main theorem, characterizing those
divisors $D$ on the cubic surface $X$ associated to a stable
Ulrich bundle of rank $r$. From the previous section we already
know the existence of some Ulrich bundles, those of rank 1 being
trivially stable. To show the existence of stable bundles of
higher ranks, we proceed inductively. Using extensions of stable
bundles of lower rank we can construct \textit{simple} Ulrich
bundles. These belong to a modular family whose dimension we can
compute. Then we show that the possible non-stable simple bundles
form families of lower dimension, so that the general bundle of
the modular family must be stable.

We start with rank 2, which is elementary (see Example
\ref{exrk2_2} below). For rank $r\geq 3$ we use the inductive
procedures of the previous section, expressing the divisor $D$ as
$D'+ T$ for a suitable twisted cubic curve $T$, and using
induction on $D'$. The computation of dimensions is achieved using
the computation of Euler characteristics done earlier (Corollary
\ref{corchi}). We also give a second proof, modeled on the ideas
of the original proof in our earlier paper \cite[5.7]{CH2}, that
does not use modular families.

As an example, we list all the divisors that correspond to rank 3
Ulrich bundles, showing which ones are stable and the dimensions
of the families. A curious byproduct of our investigation is that
when there are no stable Ulrich bundles, it is not because there
are too few semistable bundles: rather there are too many! In
these cases there are ``oversize" families of \textit{polystable}
bundles (meaning direct sums of two or more stable bundles) of
dimension bigger than the expected dimension of the moduli of
stable bundles.

\begin{ex}\label{exrk2_2}\rm\textbf{Rank 2 Ulrich bundles on the cubic surface
(cont.)}. Here we recover the results of Faenzi \cite{Faenzi} on
stable  Ulrich bundles of rank 2. We have seen in Example
\ref{exrk2} that the only possibilities for stable Ulrich bundles
of rank 2 are cases $d$, $e$, $f.$ To show the existence of stable
bundles in these cases, we consider extensions of line bundles
corresponding to twisted cubic curves
$$0 \lra \OO(T_1) \lra \EE \lra \OO(T_2) \lra 0$$
where $(T_1,T_2)=(T_A,T_C)$, $(T_B,T_{B'})$, $(T_A,T_E)$ for cases
$d,e,f$ respectively. Since $\OO(T_1)\ncong \OO(T_2)$, $
h^0(\OO(T_1-T_2))= 0,$ $h^2(\OO(T_1-T_2)) \perp
h^0(\OO(T_2-T_1-H))=0$, so $h^1(\OO(T_1-T_2))=$ $\dim
\Ext^1(\OO(T_2),\OO(T_1))=-\chi(\OO(T_1) \otimes
\OO(T_2)^{\vee})$, which is equal to $2\cdot1\cdot1-T_1.T_2$
according to Corollary \ref{corchi}. Now $T_A.T_C=3$,
$T_B.T_{B'}=4$, $T_A.T_E=5$, so $\dim \Ext^1(\OO(T_2),\OO(T_1))=$
1, 2, 3 respectively. Therefore there exist non-split extensions,
and these bundles are necessarily Ulrich, being extensions of
Ulrich bundles. Furthermore they are simple because of Lemma
\ref{lemsimple} below.

Thus in cases $d$, $e$, $f$ there are simple bundles, and we can
look at a modular family of simple bundles. Its dimension is
computed as $\dim \Ext^1(\EE,\EE)$ $=-\chi(\EE\otimes
\EE^{\vee})+1=D^2-7,$ which gives dimension 1, 3, 5 respectively
in these cases. On the other hand, the dimension of the family of
simple extensions of $\OO(T_2)$ by $\OO(T_1)$ constructed above is
$\dim \Ext^1(\OO(T_2),\OO(T_1))-1$ which we found above to be 0,
1, 2 in these cases. Since any non-stable simple Ulrich bundle
$\EE$ must be an extension of rank 1 Ulrich bundles, it will be an
extension as above for some $T_1$ and $T_2$. Now
$c_1(\EE)=D=T_1+T_2,$ and $D^2=T_1^2+2T_1.T_2+T_2^2 = 8,$ 10, 12
in cases $d$, $e$, $f$, respectively, so $T_1.T_2=3,$ 4, 5
respectively, and from the previous dimension count we conclude
that the remaining bundles in the larger modular families must be
stable. This explains the last two columns of Example \ref{exrk2}.
\end{ex}

\begin{lema}\label{lemsimple}
On a nonsingular projective variety $X$, let
$$0 \lra \FF \lra \EE \lra \GG \lra 0$$
be a non-split extension of non-isomorphic $\mu$-stable vector
bundles $\FF$, $\GG$ of the same slope. Then $\EE$ is a simple
vector bundle.
\end{lema}
\begin{proof}
If $\EE$ is not simple, then there is a map $\EE
\srel{\al}{\lra}\EE$ of rank less than $\rg \EE$. If $\al(\FF)=0$,
then $\al$ factors through $\GG$, and the map $\EE \lra \GG$
splits. Therefore $\al(\FF) \neq 0$. The composed map $\FF \lra
\EE \srel{\al}{\lra}\EE \lra \GG$ must be zero since $\FF, \GG$
are stable, not isomorphic, and of the same slope. Therefore
$\al(\FF) \subset \FF$ and so $\al{\mid}_{\FF}$ must be an
isomorphism, since $\FF$ is stable. Therefore $\al$ induces a map
from $\GG$ to $\GG$. If this is 0 then $\al$ maps $\EE$ to $\FF$
and the sequence splits again. Hence the map on $\GG$ is an
isomorphism and so $\al$ is an isomorphism, contradicting $\rg \al
<\rg \EE.$
\end{proof}

\begin{teo}\label{teo_stable}
Let $D$ satisfy the equivalent conditions of Theorem \ref{condD}
with $r \geq 2$. Then the following are equivalent.
\begin{itemize}
\item[(i)]  $D.T \geq 2r$ for all twisted cubic curves $T$, but $D\neq mD_0$ for any $m\geq 2$, where $D_0=(4; 2, 1^4,0)$.
\item[(ii)] There exist stable Ulrich bundles corresponding to $D$.
\end{itemize}
Moreover, if $D$ satisfies these conditions, then the moduli space
$M_X^s(r;c_1,c_2)$ of stable vector bundles of rank $r$ on $X$
with $c_1=D$ and $c_2=\frac{1}{2}(D^2-r)$ is smooth and
irreducible of dimension $D^2-2r^2+1$ and consists entirely of
stable Ulrich bundles.
\end{teo}

\begin{proof}

(i) $\Rightarrow$ (ii). By induction on $r$. For $r=2$, this
follows by inspection of the table of degree 6 curves in Example
\ref{exrk2}.

For $r\geq 3$, we use Proposition \ref{PropAlg} to choose a
twisted cubic curve $T$ such that $D'=D-T$ satisfies the same
conditions as $D$, namely $a'=\min\{D'.T\} \geq 2r-2$, and $D'
\neq m'D_0$ for any $m'\geq 2$. Then by induction, there exist
stable Ulrich bundles $\FF$ corresponding to $D'$. We consider
extensions
$$0 \lra \OO(T) \lra \EE \lra \FF \lra 0.$$
First we use Corollary \ref{corchi} to compute
$\chi(\FF^{\vee}(T))=2(r-1)-D'.T.$ Since $D'.T=a-1$ (by
Proposition \ref{PropAlg} we have $D.T=a$, so $D'.T=a-T^2=a-1$)
and $a \geq 2r$, this is strictly negative , and so there are
non-split extensions in $\Ext^1(\FF, \OO(T))$. These are simple by
Lemma \ref{lemsimple}. Moreover $H^0(\FF^{\vee}(T))=0$ and
$H^2(\FF^{\vee}(T))=0$ because $\FF$ is stable, and thus the
dimension of $\Ext^1(\FF, \OO(T))$ is equal to $a+1-2r> 0$.

Consider the modular family of simple bundles $\EE$ (see Remark
\ref{Rem_modfam}). We can compute its dimension as $\dim H^1(\EE
\otimes \EE^{\vee})=D^2-2r^2+1$ (Corollary \ref{cordim}). If the
general member of this modular family is not stable, then it must
be of the same type as $\EE$ above, namely an extension of a
stable rank $r-1$ bundle by a stable line bundle $\OO(T)$, because
no other semistable splitting type can specialize to this one (see
\cite[2.3.1]{HL}).

On the other hand, the dimension of the family of simple bundles
obtained by this construction is
$$\dim \{\FF\}+\dim \Ext^1(\FF,\OO(T))-1=(D')^2-2(r-1)^2+1+a-2r.$$

Since  $D'=D-T$ and $D.T=a$ by Proposition \ref{PropAlg}, this
number is
\begin{equation*}
D^2-2a+1-2r^2+4r-2+1+a-2r= D^2-2r^2-a+2r.
\end{equation*}
Since $a \geq 2r$, this number is strictly less than $D^2-2r^2+1$.
Thus the general bundle in the modular family of simple bundles
must be stable.

(ii) $\Rightarrow$ (i). For this implication we will show that if
$a<2r$, or if $D=mD_0$ for some $m\geq 2$, then there exists a
family of properly semistable polystable  bundles of dimension
$\geq D^2-2r^2+1.$ In this case there cannot be stable bundles,
because the moduli space $M^{ss,U}$ is irreducible (see Corollary
\ref{cor_irred}), and the stable bundles, if they exist, fix its
dimension as $D^2-2r^2+1$. In that case the polystable bundles
would have dimension $\leq D^2-2r^2$ in the moduli space.

We assume $a<2r$, or $D=mD_0$ for some $m\geq 2$ and proceed again
by induction on $r$. For $r=2$, looking at the table of rank 2
bundles in Example \ref{exrk2}, we have $D=2T_A$ or $D=T_A+T_B.$
These give polystable bundles, forming families of dimension 0,
while $D^2-2r^2+1$ is $D^2-8+1$ which is either -3 or -1.

For $r \geq 3$ and $a<2r$, we choose $T$ as in Proposition
\ref{PropAlg} and let $D'=D-T.$ If $a' \geq 2r-2$ then by the
above implication (i) $\Rightarrow$ (ii) there is a family of
Ulrich stable bundles $\FF$ of rank $r-1$ and dimension
$D'^2-2(r-1)^2+1.$ We take $\OO(T) \oplus \FF$ as polystable
bundles and thus we obtain a family of polystable bundles of this
same dimension which is $D^2-2r^2+4r-2a.$ On the other hand, if
$a' <2r-2$ or if $D'=m'D_0$ for some $m'\geq 2,$ then by
induction, there exists a family of polystable bundles for $D'$ of
dimension at least of this same dimension $D^2-2r^2+4r-2a$. Taking
the direct sum with $\OO(T)$ gives polystable bundles for $D$ of
at least the same dimension. Now since $a<2r$, $a \leq 2r-1$, and
this family has dimension $\geq D^2-2r^2+2.$

Finally, consider the case $D=mD_0$ with $m\geq 2$. Since $D_0$
corresponds to a 1-dimensional family of rank 2 bundles (see case
$d$ in the table in Example \ref{exrk2}), by taking direct sums of
$m$ of them we obtain an $m$-dimensional family of polystable
bundles. On the other hand, since $D_0^2=8$ and $r=2m$, in this
case $D^2-2r^2+1=8m^2-8m^2+1=1,$ so again there are no stable
bundles.

For the last statement, we recall that the Ulrich bundles form an
irreducible family (see Corollary \ref{cor_irred}) and the
corresponding moduli space is smooth if dimension $D^2-2r^2+1$
(see Corollary \ref{cordim}). Thus we have only to show that any
other stable bundle of the same rank with the same Chern classes
is an Ulrich bundle. Since the Hilbert polynomial depends only on
the Chern classes, this will be a consequence of the following
lemma, using Lemma \ref{lemaHilbpoly}.
\end{proof}

\begin{lema}
On a nonsingular cubic surface $X$, any stable vector bundle $\EE$
of rank $r$ with Hilbert polynomial
$P_{\EE}(n)=3r{{n+2}\choose{2}}$ is an Ulrich bundle, and
therefore also $\mu$-stable.
\end{lema}

\begin{proof} Since $\frac{1}{r}P_{\EE(-1)}(n)=3{{n+1}\choose{2}}$
and $P_{\OO_X}(n)=3{{n+1}\choose{2}}+1,$ it follows from the
stability of $\EE$ that  $h^0(\EE(-1))=0.$ Hence $h^0(\EE(n))=0$
for all $n<0.$

Next, we note that
$\chi(\EE^{\vee}(n+2))=\chi(\EE(-n-3))=3r{{-n-1}\choose{2}}=3r{{n+2}\choose{2}},$
so $\EE^{\vee}(2)$ is also stable satisfying the same hypothesis
as $\EE.$ It follows that $h^0(\EE^{\vee}(n))=0$ for all $n \leq
1.$ By duality, this shows $h^2(\EE(n))=0$ for all $n \geq -2.$

Now $\chi(\EE(n))=0$ for $n=-1,-2,$ so it follows that
$h^1(\EE(n))=0$ for $n=-1,-2.$ Then by Castelnuovo-Mumford
regularity, $\EE$ is regular, so $h^1(\EE(n))=0$ for all $n\geq
-2.$ The same argument applies to $\EE^{\vee}(2)$ so
$h^1(\EE^{\vee}(n))=0$ for all $n\geq 0.$ By duality this implies
$h^1(\EE(n))=0$ for all $n\leq -1.$ Thus $h^1(\EE(n))=0$ for all
$n$, and $\EE$ is an ACM sheaf.

Finally, since $h^i(\EE)=0$ for $i=1,2$ and $\chi(\EE)=3r,$ we
find $h^0(\EE)=3r,$ so $\EE$ is Ulrich by Lemma \ref{lemaH0},
hence also $\mu$-stable by Theorem \ref{teo_semist}.
\end{proof}

We thus obtain a new proof for the following statement of our
previous paper \cite[5.7]{CH2}:
\begin{cor}\label{cor_orient}
There exist stable orientable Ulrich bundles of every rank $r \geq
2$ on the cubic surface $X$ and their moduli space is smooth and
irreducible of dimension $r^2+1.$
\end{cor}

\begin{proof} Here orientable means $D=rH,$ and in this case $D.L=r$ for all $L,$
$D.T=3r$ for all $T$, and $D\neq mD_0$ for any $m$, so Theorem
\ref{teo_stable} applies.
\end{proof}

\begin{rk}\rm We give here another proof of Theorem \ref{teo_stable} (i)$\Rightarrow$(ii), without using modular families or counting dimension of
families.This proof is an adaptation of the proof of Theorem 5.3
given in \cite{CH2} and thus validates the idea of that proof if
not all of its details.

As in the previous proof, we use induction on $r$, the case $r=1$
being trivial (the only Ulrich bundles are $\OO(T)$, $T$ a twisted
cubic, and these are stable). For $r=2$ we use Example
\ref{exrk2_2} or \cite{Faenzi}. For $r\geq 3$, given a divisor $D$
satisfying the conditions of Theorem \ref{teo_stable}(i), we
choose a $T$ such that $D'=D-T$ satisfies the same conditions by
\ref{PropAlg}(b), so by induction there exists a stable bundle
$\FF$ of rank $r-1$ corresponding to $D'$. Then as in the proof of
Theorem \ref{teo_stable} we consider extensions
$$0\lra \OO(T) \lra \EE \lra \FF \lra 0.$$
Since there is an irreducible family containing all Ulrich bundles
corresponding to $r$ and $D$, if there are no stable ones, then
the general one would have to be an extension of some stable
bundle $\FF$ of rank $r-1$ by $\OO(T)$: for no other splitting
type could specialize to this one.

By Proposition \ref{prop_irredfam}, any one of these Ulrich
bundles can be represented as an extension
\begin{equation*}
0 \lra \OO_X^{r-1} \lra \EE \lra \JJ_Z(D) \lra 0.
\end{equation*}
for some subscheme $Z$ with $h^0(\JJ_Z(D-H)) = 0$. This last
condition is an open condition on $Z$, so there will exist
extensions $\EE$ corresponding to a general set of points $Z$, and
these will  also be Ulrich, since the Ulrich condition is open
(see section 2 above). For such an $\EE$ we will show there are no
non-zero morphisms from $\OO(T)$ to $\EE.$ We want to compute
$\Hom(\OO(T), \EE)=H^0(\EE(-T))).$ So we consider the sequence
\begin{equation*}
0 \lra \OO(-T)^{r-1} \lra \EE(-T) \lra \JJ_Z(D-T) \lra 0.
\end{equation*}
Now $h^0(\OO(-T))=0.$ So to show $h^0(\EE(-T))=0$ it is sufficient
to show $h^0(\JJ_Z(D-T))=0.$ Since we can choose the points $Z$ in
general position, it will be sufficient to show that $n =\sharp
Z\geq h^0(\OO(D-T)).$ We compute both numbers. From Proposition
\ref{cor_nonsing} we find $n=c_2(\EE)=\frac{1}{2}(D^2-r).$ On the
other hand, since $D'=D-T$ can be represented by an irreducible
nonsingular curve (see Proposition \ref{cor_nonsing}), we have
$h^i(\OO(D'))=0$ for $i=1,2,$ and so $h^0(\OO(D'))$ can be
computed by Riemann-Roch:
\begin{align*}
h^0(\OO(D')) = & \frac{1}{2}(D'+H)D'+1\\
=& \frac{1}{2}(D^2-2D.T+T^2+3(r-1))+1\\
  =&\frac{1}{2}(D^2-2D.T+3r).
\end{align*}
Our hypothesis on $D$ says $D.T\geq 2r$ for all twisted cubic
curves $T$, and so $h^0(\OO(D'))\leq \frac{1}{2}(D^2-r)=n.$ Hence
$n$ general points will make $h^0(\JJ_Z(D-T))=0$ so
$h^0(\EE(-T))=0,$ and the general $\EE$ must be stable.
\end{rk}

\begin{ex}\rm
\textbf{Rank 3 Ulrich bundles on the cubic surface.} To illustrate
our main theorem, we classify Ulrich bundles of rank 3. To do
this, in the following table we list all divisor classes
representing an irreducible nonsingular curve of degree $3r=9$
(see Proposition \ref{cor_nonsing}), in standard form and then
retain only those that satisfy $0\leq D.L\leq 2r=6$ for all lines
$L$. These are exactly the divisors $D$ that can correspond to
Ulrich bundles.

\begin{center}
\begin{tabular}{|c|c|c|c|c|}\hline
D&$D^2$&$\sum T_i$&$D^2-2r^2+1$&comments\\
\hline
(3;0,0,0,0,0,0)&9&$3A$&-8 & $\oplus \OO(T_i)$\\
(4;1,1,1,0,0,0)&13&$2A+B$&-4& $\oplus \OO(T_i)$\\
(5;1,1,1,1,1,1)&19&$A+B+B'$&2& 3-dim. polystable\\
(5;2,1,1,1,1,0)&17&$2A+C$&0& 1-dim. polystable\\
(5;2,2,1,1,0,0)&15&$A+B+B''$&-2& $\oplus \OO(T_i)$\\
(6;2,2,2,1,1,1)&21&$2B+B'$&4& $\exists$ stable\\
(6;2,2,2,2,1,0)&19&$A+B'+C$&2&$\exists$ stable\\
(6;3,2,1,1,1,1)&19&$A+B'+C$&2&$\exists$ stable\\
(7;2,2,2,2,2,2)&25&$2A+E$&8& $\exists$ stable\\
(7;3,2,2,2,2,1)&23&$A+B+D$&6& $\exists$ stable\\
(9;3,3,3,3,3,3)&27&$B+B'+E(=3H)$&10& $\exists$ stable\\
\hline
\end{tabular}
\end{center}

Note that by our main Theorem \ref{teo_stable} there exist stable
bundles in the last six cases, and the corresponding moduli spaces
have dimension $D^2-2r^2+1$. In the third and four row of the
table if there were stable bundles, the moduli space would have
dimension 2, 0 respectively. In these cases there are no stable
bundles, but we have ``oversize" families of polystable bundles of
dimensions 3, 1 respectively. In each case they are direct sums of
a rank 1 bundle corresponding to a twisted cubic curve and a
stable rank 2 bundle corresponding to $D'=(4;1^6)$ and
$(4;2,1^4,0)$ respectively. In particular, the necessary condition
$D^2-2r^2+1 \geq 0$ is not sufficient for the existence of stable
bundles.
\end{ex}

\begin{rk}\rm In proving the nonexistence of stable bundles in case $a<2r$,
we constructed oversize families of polystable bundles. One can
see from the proof, inductively, that the polystable bundles we
construct there are all direct sums of stable bundles of ranks 1
or 2. One might ask if there are other types of polystable bundles
that are not specializations of stable bundles. That this does not
happen is a consequence of the following corollary.
\end{rk}

\begin{cor}
Let $\FF$ and $\GG$ be stable Ulrich bundles on the cubic surface
$X$, of ranks $s,t\geq 2.$ Then the polystable bundle $\FF\oplus
\GG$ is a specialization of a stable bundle unless $s=t=2$ and
both $\FF$ and $\GG$ are associated to the same divisor class
$D_0=(4;2,1^4,0).$
\end{cor}

\begin{proof}
Since $\FF$ and $\GG$ are both stable, corresponding to divisors
$C,D$, say, by Theorem \ref{teo_stable} we know that $C.T \geq 2s$
and $D.T\geq 2t$ for all twisted cubic curves $T.$ Therefore
$(C+D).T \geq 2(s+t)$ for all $T$, and we conclude by Theorem
\ref{teo_stable} that $C+D$ corresponds to a family of stable
bundles that will specialize to $\FF \oplus \GG,$ unless
$C+D=mD_0$ for some $m \geq 2.$

In this case, we write $D_0$ in standard form as $(4;2,1^4,0),$ so
that $mD_0$ is $(4m;2m,m^4,0),$ and the rank $r=s+t$ of the
corresponding Ulrich bundle is $2m.$ Having chosen a basis for
$\Pic X$ so that $D_0$ is in standard form, we cannot assume that
$C,D$ are in standard form. So let  $C=(a;b_1, \dots, b_6)$,
$D=(a';b_1', \dots, b_6').$ Then $a+a'=4m$, $b_1+b_1'=2m$,
$b_i+b_i'=m$ for $i=2,3,4,5$, and $b_6=b_6'=0$ since both are
$\geq 0$ (because $C.L \geq 0$, $D.L \geq 0$ by Theorem
\ref{condD}) and their sum is 0.

Next, since $C.T \geq 2s$ for all $T$, taking $T=T_A=(1;0^6)$, we
find that $a \geq 2s$. Similarly, $a' \geq 2t$. But $a+a'=4m=2r$,
so we must have equality in both cases. On the other hand, taking
$T=T_B=(2;1^3,0^3)$ or one of its permutations, we find that
$2a-b_i-b_j-b_k \geq 2s$ for any three $i,j,k \in \{1, \dots,6\}.$
Since $a=2s,$ we find $a\geq b_i+b_j+b_k.$ This shows that $C$ is
in standard form except for the ordering of the $b_i,$ which we do
not know yet.

Consider the line $L=G_1$ (notation of \cite[V 4.9]{AG}.) We know
$C.L \leq 2s$ and $D.L\leq 2t$ by Theorem \ref{condD} (iii). On
the other hand $mD_0.L=4m=2r.$ Therefore again we have equality in
both cases. Thus $C.L=2a-\sum_{i=2}^6 b_i=2s,$ and $\deg
C=3a-\sum_{i=1}^6 b_i=3s,$ so we find $b_1=a-s=s,$  and
$C=(2s;s,b_2, \dots, b_5,0),$ and similarly for  $D.$ Then from
$\deg C=3s$ we find $b_2+b_3+b_4+b_5=2s.$ On the other hand, from
$a\geq b_1+b_i+b_j,$ we find $b_i+b_j \leq s$ for any $i\neq j$
$\in \{2,\dots, 5\}.$ This implies $b_i+b_j=s$ for all $i\neq j$,
so $b_i=s/2$ for $i=2, \dots,5$ and $C=\frac{1}{2}sD_0$. Similarly
we get $D=\frac{1}{2}tD_0.$

Finally, since we have assumed $C,D$ correspond to stable bundles,
by Theorem \ref{teo_stable} we must have $s=t=2$ and $C=D=D_0.$
\end{proof}

\section{Stable Ulrich bundles on the cubic threefold}
In this section we construct stable Ulrich bundles of all ranks
$r\geq 2$ on a general cubic threefold $Y$ in $\PP^4$, and we show
that the corresponding moduli space is smooth of the expected
dimension $r^2+1$. Stable bundles of rank 2 are well known
\cite{Beauville}, \cite{Beauvillecubic}, \cite{AC},
\cite{Markushevich}. We construct stable rank 3 Ulrich bundles
using curves whose existence is proven by a \texttt{Macaulay2}
computation due to Gei{\ss} and Schreyer (see Appendix). Then we
use a method analogous to the case of surfaces, creating simple
bundles as extensions of stable bundles of lower rank, and
counting dimensions to show that the general bundles in a modular
family must be stable. Finally, for each $r$, we show that there
is at least one component of the moduli space of stable Ulrich
bundles of rank $r$ on $Y$ for which the restriction map to the
moduli of stable rank $r$ bundles on the hyperplane section, a
cubic surface, is generically étale and dominant.

For completeness, we include a proof of existence of stable rank 2
Ulrich bundles on $Y$.

\begin{prop}\label{prop_rk2_3fold} On a nonsingular cubic threefold $Y$, there exist stable rank 2 Ulrich bundles on $Y$ with first Chern class
$c_1=2H$, where $H$ is the hyperplane class, and $c_2=5.$ The
moduli space of these bundles is smooth of dimension 5.
\end{prop}
\begin{proof} If $\EE$ is such a bundle, a general section will
vanish along a quintic elliptic curve $C$ in $Y$, not contained in
any hyperplane. Thus, to construct such bundles using the Serre
correspondence, we need to show the existence in $Y$ of a
nonsingular elliptic curve $C$ of degree 5 that is not contained
in any hyperplane section. Let $H$ be a general hyperplane section
of $Y$. This is a nonsingular cubic surface, and on it we can find
a quintic elliptic curve $C_0$, for example in the divisor class
$(3;1^4,0^2)$. This curve has self-intersection $C_0^2=5.$ There
is an exact sequence for the normal bundle of $C_0$ in $Y$
$$0\lra \NN_{C_0/H}\lra \NN_{C_0/Y}\lra \NN_{H/Y}|_{C_0}\lra 0.$$
The first of these is $\OO_{C_0}(C_0)$, the third is
$\OO_{C_0}(1).$ Both have degree 5, $h^0=5$ and $h^1=0$. Hence
$h^0(\NN_{C_0/Y})=10$ and $h^1(\NN_{C_0/Y})=0.$ Thus the Hilbert
scheme of quintic elliptic curves in $Y$ is smooth of dimension 10
at the point corresponding to $C_0$.

Let us count how many quintic elliptic curves there are contained
in hyperplane sections $H$ of $Y.$ The choice of $H$ is four
parameters, and the dimension of the linear system $|C_0|$ on $H$
is five, so there is a 9-dimensional family of quintic elliptic
curves in hyperplanes of $Y$. We conclude that a general quintic
elliptic curve on $Y$ is not contained in any hyperplane section.
Let $C$ be one of these.

We apply the Serre construction \cite[1.1]{Harts_stbldes} to
obtain a bundle of rank 2 as an extension
$$0\lra \OO_Y \lra \EE \lra \JJ_C(2)\lra 0.$$
The extension is determined by an element of
$\Ext^1(\JJ_C(2),\OO_Y) \cong \Ext^2(\OO_C(2), \OO_Y)$ $\cong
H^0(\omega_C).$ Since $C$ is an elliptic curve, $\omega_C\cong
\OO_C$ and there is just one choice of section that is nowhere
vanishing. Hence $\EE$ is locally free of rank 2. Since $C$ is an
ACM curve in $Y$ (cf. \cite[3.4]{Hargor}), it follows that $\EE$
is an ACM bundle on $Y$. Indeed, $C$ being ACM means
$H^1_{\ast}(\JJ_C)=0$, and so $H^1_{\ast}(\EE)=0.$ Since $\EE$ has
rank 2, its dual $\EE^{\vee}$ is isomorphic to $\EE(-c_1),$ and
then by duality $H^2_{\ast}(\EE),$ dual to
$H^1_{\ast}(\EE^{\vee}(-1)),$ is also zero.  Since $C$ is not
contained in a hyperplane, $h^0(\EE(-1))=0.$ We see that
$h^0(\EE)=6$, so $\EE$ is an Ulrich bundle. Since there are no
Ulrich bundles of rank 1 on $Y$, by Theorem \ref{teo_semist} it
cannot be properly semistable, so it is stable. To show that the
moduli space is smooth of dimension 5, we just compute $H^1(\EE
\otimes \EE^{\vee})=5$ and $H^2(\EE \otimes \EE^{\vee})=0$. This
is elementary (left to the reader).
\end{proof}

\begin{rk}
\rm In fact, Beauville shows that this moduli space is isomorphic
to an open subset of the intermediate Jacobian of the cubic
threefold \cite[5.2]{Beauvillecubic} and therefore is irreducible,
but we will not make use of this fact.
\end{rk}

\begin{defi}
For the rest of this section, a general cubic threefold in $\PP^4$
will denote a cubic hypersurface in a suitable Zariski open subset
of the Hilbert scheme of cubic hypersurfaces in $\PP^4$ over an
algebraically closed field either of characteristic zero, or of
prime characteristic $p$, except possibly for finitely many values
of $p$.
\end{defi}

\begin{prop}\label{prop_rk3_3fold}
On a general cubic threefold $Y$ in $\PP^4$, there exist stable
Ulrich bundles of rank 3.
\end{prop}

\begin{proof}
In the Appendix it is shown that $Y$ contains smooth ACM curves
$C$ of degree 12 and genus 10 with the additional property that
$\omega_C(-1)$ has two global sections that generate the graded
module $H^0_{\ast}(\omega_C)$ over the homogeneous coordinate ring
$S$ of $\PP^4$.

We use the Serre construction to create an exact sequence
$$0\lra \OO_Y^2 \lra \EE \lra \JJ_C(3)\lra0 $$
by using two generators of
$H^0(\omega_C(-1))\cong \Ext^1(\JJ_C(3),\OO_Y)$ to define the
extension. Since $C$ is an ACM curve, it follows that
$H^1_{\ast}(\EE)=0.$ The dual sequence is
$$0\lra \OO_Y(-3)\lra \EE^{\vee}\lra \OO_Y^2 \lra \omega_C(-1)\lra 0,$$
and now the fact that the two sections of $\omega_C(-1)$ generate
the module $H^0_{\ast}(\omega_C)$ shows that
$H^1_{\ast}(\EE^{\vee})\cong H^2_{\ast}(\EE(-3)).$ Thus $\EE$ is
an ACM bundle. Again using the fact that $C$ is an ACM curve, we
compute $h^0(\OO_C(2))=15,$ so $h^0(\JJ_C(2))=0,$ and
$h^0(\EE(-1))=0;$ furthermore $h^0(\OO_C(3))=27$, so
$h^0(\JJ_C(3))=0$ and $h^0(\EE)=9=3r.$ Hence by Lemma \ref{lemaH0},
$\EE$ is a rank 3 Ulrich bundle on $Y.$ It is necessarily stable,
because there are no rank 1 Ulrich bundles on $Y.$
\end{proof}

\begin{rk}\rm
The existence of rank 3 Ulrich bundles on $Y$ was announced
earlier \cite[Example 4.4]{AC}, but the proof given there is
incorrect because the ACM curve of degree 12, genus 10 that they
used, Gorenstein linked to a conic, does not satisfy the
additional condition that the sections of $\omega_C(-1)$ should
generate $H^0_{\ast}(\omega_C).$
\end{rk}

\begin{prop}\label{Prop_chi_3fold}
Let $\EE$, $\FF$, be Ulrich bundles on the cubic threefold $Y.$
Then
\begin{itemize}
\item[a)] $\chi(\EE \otimes \FF^{\vee}(-1))=0$
\item[b)] $\chi(\EE \otimes \FF^{\vee})=\chi(\EE_H\otimes \FF_H^{\vee})=0,$ where $H$ is a general hyperplane section
\item[c)] $H^i(\EE \otimes \FF^{\vee})=0$ for $i=2,3.$
\end{itemize}
\end{prop}

\begin{proof}
a) By Serre duality, $\chi(\EE \otimes
\FF^{\vee}(-1))=-\chi(\EE^{\vee} \otimes \FF(-1)).$ Let
$\EE^{'}=\EE^{\vee}(2)$ and $\FF'=\FF^{\vee}(2).$ These are Ulrich
bundles, by Lemma \ref{lemaHn} (iv). We can rewrite our equation
as $\chi(\EE \otimes \FF'(-3))=-\chi(\EE' \otimes \FF(-3)).$ Note
that $\chi$ depends only on the Chern classes of the bundles in question; the Chern classes of a tensor product are
determined by those of the two factors; and the Chern classes of an Ulrich bundle are determined by its rank (Lemma \ref{lemaHilbpoly}).
Since $\EE$ and $\EE'$ have the same rank, and $\FF$ and $\FF'$ have the same rank, $\EE \otimes \FF'$ and $\EE' \otimes \FF$ have the same Chern classes,
so $\chi(\EE \otimes \FF'(-3))=\chi(\EE' \otimes \FF(-3)).$ Combining with the above, both must be zero.

b) This is a direct consequence of a), tensoring $\EE \otimes
\FF^{\vee}$ with the exact sequence
$$0\lra \OO_Y(-1)\lra \OO_Y \lra \OO_H\lra 0.$$

c) Since $\EE$ is Ulrich, it has a linear resolution over $\PP^4$ \cite[3.7]{CH2},
$$0\lra \OO_{\PP^4}(-1)^{3r} \lra \OO_{\PP^4}^{3r} \lra \EE \lra 0.$$

Tensoring with $\FF^{\vee}$, we get a right exact sequence on $Y,$
$$\FF^{\vee}(-1)^{3r}\lra \FF^{\vee 3r} \lra \EE \otimes \FF^{\vee}\lra 0.$$
Now $\FF^{\vee}(2)$ is an Ulrich bundle by Lemma \ref{lemaHn}(iv), so $\FF^{\vee}$ and $\FF^{\vee}(-1)$ have no cohomology.
It follows from cohomology sequences on $Y$, that $H^2(\EE\otimes \FF^{\vee})=H^3(\EE \otimes \FF^{\vee})=0.$
\end{proof}

\begin{teo}\label{teo_3fold}
For any $r\geq 2$, the moduli space of stable rank $r$ Ulrich bundles on a general cubic threefold $Y$ in $\PP^4$ is non-empty and smooth of dimension
$r^2+1.$
\end{teo}

\begin{proof}
For any Ulrich bundle $\EE,$ we have $H^2(\EE\otimes \EE^{\vee})=0$ by Proposition \ref{Prop_chi_3fold} c), and this implies the smoothness of the moduli space.
Furthermore, by the same Proposition, $\chi(\EE\otimes \EE^{\vee})=\chi(\EE_H\otimes \EE_H^{\vee})=-r^2$ by Corollary \ref{corchi}, since $c_1(\EE)=rH$.
For $\EE$ stable or simple we have $h^0(\EE\otimes \EE^{\vee})=1,$ and $h^2(\EE\otimes \EE^{\vee})=h^3(\EE\otimes \EE^{\vee})=0$
by Proposition \ref{Prop_chi_3fold} c), so $h^1(\EE\otimes \EE^{\vee})=r^2+1$ is the dimension of the moduli space.

It remains to show the existence. We proceed by induction on $r,$ the cases $r=2,3$ having been done above. So let $r\geq 4,$
and choose $\EE$ stable of rank 2, and $\FF$ stable of rank $r-2,$ different from $\EE.$ Then $h^i(\EE\otimes \FF^{\vee})=0$ for
$i=0,2,3,$ so $h^1(\EE\otimes \FF^{\vee})=-\chi(\EE\otimes \FF^{\vee})=-\chi(\EE_H\otimes \FF_H^{\vee})=2(r-2)$ by Corollary \ref{corchi}.
In particular, this number is positive, so there exist nonsplit extensions
$$0\lra \EE \lra \GG \lra \FF \lra 0,$$
and the new bundle $\GG$ will be a simple Ulrich bundle of rank $r$ (see Lemma \ref{lemsimple}). We consider the modular family of these simple bundles,
which will be smooth of dimension $r^2+1$ by the above observations.

If the general simple bundle in this family is not stable, it must have the same splitting type as the ones just constructed. However, the dimension of the family
of extensions above is

$$\begin{array}{rl}
& \dim\{\EE\}+\dim\{\FF\}+\dim(\Ext^1(\FF, \EE))-1\\
=&2^2+1+(r-2)^2+1+2(r-2)-1\\
=& r^2-2r+5.
\end{array}
$$
Since $r\geq 4,$ this number is strictly less than $r^2+1.$ We conclude that the general simple bundle of rank $r$ is stable, so stable bundles exist.
\end{proof}

Our next goal is to study the restriction map from Ulrich bundles on $Y$ to bundles on a hyperplane section $H.$ We will show in many cases that there is an
open set of stable bundles on $Y$ that restricts by an étale dominant map to stable bundles on $H.$

\begin{prop}\label{prop_etale_3fold}
Suppose that $\EE$ is a stable Ulrich bundle of rank $r$ on $Y$ with the property that $H^i(\EE \otimes \EE^{\vee}(-1))=0$ for all $i$
(in which case we say $\EE \otimes \EE^{\vee}(-1)$ has no cohomology). Then the restriction map from bundles on $Y$ to bundles on the general
hyperplane section $H$ induces an étale dominant map from an open subset of a modular family of stable rank $r$ Ulrich bundles on $Y$ to
a modular family of stable rank $r$ Ulrich bundles on $H.$
\end{prop}

\begin{proof}
First we recall that the restriction of an Ulrich bundle $\EE$ on $Y$ to $H$ is also an Ulrich bundle $\EE_H$ (Lemma \ref{lemaHn}). The condition that
$\EE \otimes \EE^{\vee}(-1)$ has no cohomology implies that $H^i(\EE \otimes \EE^{\vee}) \lra H^i(\EE_H \otimes \EE_H^{\vee})$ is an isomorphism for all $i.$
If $\EE$ is stable, then $h^0(\EE \otimes \EE^{\vee})=1$ and $h^1(\EE \otimes \EE^{\vee})=r^2+1.$ Therefore the same is true for $\EE_H,$
hence $\EE_H$ is simple.
The condition that $\EE \otimes \EE^{\vee}(-1)$ has no cohomology is an open condition, so we obtain a morphism from an open subset of a
modular family of stable bundles on $Y$ to a modular family of simple bundles on $H$. This map induces an isomorphism on the Zariski tangent spaces at the point $\EE,$
hence is étale and dominant in a neighborhood of $\EE$. The modular family on $H$ contains a nonempty open subset of stable bundles
(Corollary \ref{cor_orient}), and the inverse image of this open set gives an open set of stable bundles on $Y$ which restricts by an étale dominant
map to stable bundles on $H$, as required.
\end{proof}

\begin{lema}\label{lem_normalbdl}
Let $\EE$ be a rank $r$ Ulrich bundle on $Y$ corresponding to a nonsingular curve $C$ via the exact sequence
$$0\lra \OO_Y^{r-1}\lra \EE \lra \JJ_C(r)\lra 0.$$
Then
\begin{itemize}
\item[a)] $H^i(\EE \otimes \EE^{\vee}(-1))=0$ for $i=0,3$
\item[b)] $H^i(\EE \otimes \EE^{\vee}(-1)) \cong H^{i-1}(\mathcal{N}_{C/Y}(-1))$ for $i=1,2$ where $\mathcal{N}_{C/Y}$ is the normal bundle of $C$ in $Y.$
\end{itemize}
\end{lema}
\begin{proof}
a) Since $\EE$ is a quotient of $\OO_Y^{3r}$, it follows that  $\EE \otimes \EE^{\vee}(-1)$ is a quotient of $ \EE^{\vee}(-1)^{3r}.$
But since $\EE^{\vee}(2)$ is an Ulrich sheaf, this sheaf has no $H^3,$ and it follows that $H^3(\EE \otimes \EE^{\vee}(-1))=0.$ By duality also
$H^0(\EE \otimes \EE^{\vee}(-1))=0.$

b) Tensoring the exact sequence above with $\EE^{\vee}(-1)$ we get
$$0\lra \EE^{\vee}(-1)^{r-1}\lra \EE \otimes \EE^{\vee}(-1) \lra \EE^{\vee}\otimes \JJ_C(r-1)\lra 0.$$
Now $\EE^{\vee}(2)$ is Ulrich, so $\EE^{\vee}(-1)$ has no cohomology, and $H^i(\EE \otimes \EE^{\vee}(-1))\cong H^i(\EE^{\vee}\otimes \JJ_C(r-1))$
for all $i$.

The dual of the sequence for $\EE$, twisted by $r-1$ is
$$0\lra \OO_Y(-1)\lra \EE^{\vee}(r-1) \lra \OO_Y(r-1)^{r-1} \lra \omega_C(1)\lra 0.$$
Tensoring with $\OO_C$ this gives
\begin{equation}\label{seq_1_sect5}
\EE_C^{\vee}(r-1) \lra \OO_C(r-1)^{r-1} \lra \omega_C(1)\lra 0.
\end{equation}
On the other hand, tensoring the original sequence with $\OO_C$ gives
$$\OO_C^{r-1}\lra \EE_C\lra \JJ_C/\JJ_C^2(r)\lra 0,$$
and since these are locally free sheaves on $C,$ we can dualize and twist by $r-1$ to get
\begin{equation}\label{seq_2_sect5}
0\lra \mathcal{N}_{C/Y}(-1)\lra \EE_C^{\vee}(r-1) \lra \OO_C(r-1)^{r-1}.
\end{equation}
The map in the middle of sequences  (\ref{seq_1_sect5}) and (\ref{seq_2_sect5}) is the same, so we can combine to get
\begin{equation*}
0\lra \mathcal{N}_{C/Y}(-1)\lra \EE_C^{\vee}(r-1) \lra \OO_C(r-1)^{r-1}\lra \omega_C(1)\lra 0.
\end{equation*}
Finally, we tensor the sequence $0\ra \JJ_C\ra \OO_Y \ra \OO_C\ra 0$ with $\EE^{\vee}(r-1)$ and put our sequences together in the following diagram:
$$\begin{array}{ccccccc}
&  &  & 0 &  & 0 &  \\
&  &  & \downarrow &  & \downarrow &  \\
&  &  & \OO_Y &  & \mathcal{N}_{C/Y}(-1) &  \\
   &  &    & \downarrow &   & \downarrow &  \\
 0 \lra  & \EE^{\vee}\otimes \JJ_C(r-1) & \lra   &  \EE^{\vee}(r-1) & \lra   & \EE_C^{\vee}(r-1) & \lra 0  \\
& \downarrow &    & \downarrow \alpha &   & \downarrow \beta&  \\
0 \lra  & \JJ_C(r-1)^{r-1} & \lra   &  \OO_Y(r-1)^{r-1} &  \srel{\gamma}{\lra}   & \OO_C(r-1)^{r-1} & \lra 0  \\
&  &    & \downarrow &   & \downarrow &  \\
&  &    & \omega_C(1) &  \srel{\cong}{\lra}   &  \omega_C(1)&  \\
&  &    & \downarrow &   & \downarrow &  \\
&  &    & 0 &   & 0 &
\end{array}$$
Since $\EE$ is Ulrich, it follows that $h^i(\JJ_C(r-1))=0$ for all $i.$
Hence the map $\gamma$ induces an isomorphism on cohomology, and therefore  an isomorphism on cohomology of $\im \alpha$ to $\im \beta.$
Since $\OO_Y(-1)$ has no cohomology, $\alpha$ also induces an isomorphism of cohomology from $\EE^{\vee}(r-1)$ to $\im \alpha.$
It follows that $H^0(\EE^{\vee} \otimes \JJ_C(r-1))=0$ and $H^1(\EE^{\vee} \otimes \JJ_C(r-1))\cong H^0(\mathcal{N}_{C/Y}(-1)).$
Furthermore, $\beta$ induces a surjective map on $H^0,$ and $H^1(\im \beta)=0,$ so $H^1(\mathcal{N}_{C/Y}(-1)) \ra H^1(\EE_C^{\vee}(r-1)) \ra
H^2(\EE^{\vee}\otimes \JJ_C(r-1))$ are all isomorphisms. Combining with the isomorphisms already proved above gives the desired conclusion b).
\end{proof}

\begin{cor}\label{cor_nocohrk23}
There exist rank 2 and rank 3 stable Ulrich bundles $\EE$ on a general cubic threefold $Y$ such that $\EE \otimes \EE^{\vee}(-1)$ has no cohomology.
\end{cor}
\begin{proof}
Indeed, this follows from Lemma \ref{lem_normalbdl} using the
computations of Gei{\ss} and Schreyer in the Appendix, since they
constructed ACM curves $C$ of degree 5 and genus 1 and of degree
12 and genus 10 having $H^{i}(\mathcal{N}_{C/Y}(-1))=0$ for
$i=0,1,$ curves that give rise to stable rank 2 and rank 3 Ulrich
bundles as in Propositions \ref{prop_rk2_3fold} and
\ref{prop_rk3_3fold}.
\end{proof}

\begin{prop}\label{prop_nocohomol_3fold}
For each $r \geq 2$, there is a stable rank $r$ Ulrich bundle
$\EE$ on the general cubic threefold $Y$ such that $\EE \otimes
\EE^{\vee}(-1)$ has no cohomology.
\end{prop}
\begin{proof}
Let $\EE_0$ be a stable rank 2 Ulrich bundle such that $\EE_0
\otimes \EE_0^{\vee}(-1)$ has no cohomology (Corollary
\ref{cor_nocohrk23}). We will prove by induction the following
statement

($\ast$) For each $r\geq 2$ there is a stable rank $r$ Ulrich
bundle $\FF$ on $Y$, $\FF\ncong \EE_0$, such that $\FF \otimes
\FF^{\vee}(-1)$ and $\FF \otimes \EE_0^{\vee}(-1)$ have no
cohomology.

The condition of having no cohomology is an open condition, so for $r=2$ we can take $\FF$ to be a deformation of $\EE_0.$ Then by semicontinuity, both
$\FF \otimes \FF^{\vee}(-1)$ and
$\FF \otimes \EE_0^{\vee}(-1)$ will have no cohomology.

For $r=3,$ we make use of the third section of the Appendix, which
shows that on a general cubic threefold $Y$, there are curves $E$
and $C$ as in the earlier theorems \ref{elliptic} and \ref{main}
respectively, such that if $\EE_0$ (changing notation) is the rank
2 bundle corresponding to $E:$
$$0\lra \OO_Y\lra \EE_0\lra \JJ_{E/Y}(2)\lra 0,$$
then we have also the additional property that $H^i(\EE_0 \otimes
\JJ_{C/Y} )=0$ for $i=1,2.$ Let $\FF$ be a stable rank 3 bundle
corresponding to $C$ (as in Proposition \ref{prop_rk3_3fold}):
$$0\lra \OO_Y^2\lra \FF\lra \JJ_{C/Y}(3)\lra 0.$$
Tensoring with $\EE_0^{\vee}(-1)$ we have
$$0\lra \EE_0^{\vee}(-1)^2\lra \FF \otimes \EE_0^{\vee}(-1)\lra \EE_0^{\vee}(-1) \otimes\JJ_{C/Y}(3)\lra 0.$$
Now $\EE_0$ has rank 2, so $\EE_0^{\vee}\cong \EE_0(-2).$ Thus
$\EE_0^{\vee}(-1)\cong \EE_0(-3),$ which has no cohomology.
Furthermore, since $\FF$ and $\EE_0$ are distinct stable bundles,
already $H^0(\FF\otimes \EE_0^{\vee})=0,$ so also $H^0(\FF\otimes
\EE_0^{\vee}(-1))=0,$ and by duality also $H^3(\FF\otimes
\EE_0^{\vee}(-1))=0.$ To show therefore that $\FF\otimes
\EE_0^{\vee}$ has no cohomology, we have only to check the
vanishing of $H^i$ for $i=1,2.$ Since $\EE_0^{\vee}(-1)$ has no
cohomology, the groups on question are isomorphic to $H^i(
\EE_0^{\vee}(-1)\otimes \JJ_{C/Y}(3))=H^i( \EE_0\otimes
\JJ_{C/Y}),$ and these are zero by Proposition
\ref{cohomologyExtension} of the Appendix. We have shown that
$\FF\otimes \FF^{\vee}(-1)$ has no cohomology earlier (Corollary
\ref{cor_nocohrk23}). (Note that at this step we have redefined
the rank 2 bundle $\EE_0$ chosen before, but we can just as well
use this one from the beginning.)

For $r\geq 4,$ choose by the induction hypothesis a stable bundle
$\FF_0$ of rank $r-2,$ different from $\EE_0,$ such that $\FF_0
\otimes \FF_0^{\vee}(-1)$ and $\FF_0 \otimes \EE_0^{\vee}(-1)$
have no cohomology. As in the proof of existence of stable bundles
(Theorem \ref{teo_3fold}), consider an extension
$$0\lra \EE_0\lra \GG \lra \FF_0\lra 0.$$
Then $\GG$ will be simple of rank $r.$ Tensoring with
$\EE_0^{\vee}(-1)$ and using our hypotheses on $\EE_0$ and
$\FF_0$, we see that $\GG\otimes \EE_0^{\vee}(-1)$ has no
cohomology. Similarly tensoring with $\FF_0^{\vee}(-1)$ we find
that $\GG\otimes \FF_0^{\vee}(-1)$ has no cohomology. (Note that
$\EE_0\otimes \FF_0^{\vee}(-1)=(\FF_0\otimes
\EE_0^{\vee}(-1))^{\vee}\otimes \omega_Y$ so by Serre duality it
has also no cohomology.) Now tensor $\GG(-1)$ with the dual
sequence
$$0\lra \FF_0^{\vee}\lra \GG^{\vee} \lra \EE_0^{\vee}\lra 0$$
to see that $\GG \otimes \GG^{\vee}(-1)$ has no cohomology.
Finally, as in Theorem \ref{teo_3fold} we can deform $\GG$ into a
stable bundle, call it $\FF$,  and by semicontinuity it will
satisfy $\FF \otimes \FF^{\vee}(-1)$ and $\FF \otimes
\EE_0^{\vee}(-1)$ have no cohomology.
\end{proof}

\begin{cor}\label{finalCorollary}
For each $r\geq 2$, there is a nonempty open set of a modular
family of stable rank $r$ Ulrich bundles on the general cubic
threefold $Y$ restricting by an étale dominant map to a modular
family of stable rank $r$ bundles on a hyperplane section $H.$
\end{cor}
\begin{proof}
Follows from Propositions \ref{prop_nocohomol_3fold} and
\ref{prop_etale_3fold}.
\end{proof}


\bibliographystyle{amsplain}

\begin{thebibliography}{6}
\bibitem{Geiss} F. Gei\ss, The unirationality of Hurwitz-spaces of $6$-gonal curves of genus $\leq 40$, to be completed
\end{thebibliography}
\providecommand{\href}[2]{#2}

\appendix
%
\section{ACM Curves of small degree on cubic Threefolds}
%
%
\begin{center}
{by Florian Gei\ss}\footnote{Mathematik und Informatik,
Universit\"at des Saarlandes, Campus E2 4, 66123 Saarbr\"ucken,
Germany. email: fg@math.uni-sb.de. The first author is supported
by DFG grant Schr 307/5-1 of the second author within the priority
program SSP 1409.} { and Frank-Olaf Schreyer}\footnote{Mathematik
und Informatik, Universit\"at des Saarlandes, Campus E2 4, 66123
Saarbr\"ucken, Germany. email: schreyer@math.uni-sb.de}
\end{center}

\begin{abstract}
In this note we prove the following: A general elliptic normal
curve $E$ of degree $5$ on a general cubic threefold $X\subset
\mathbb{P}^4$ over an algebraically closed field of characteristic
$0$ has a twisted normal bundle which splits as
$\mathcal{N}_{E/X}(-1)\cong L \oplus L^{-1}$ with $L \in
\mathrm{Pic}^0(E)$, $L\not\cong \mathcal{O}_{E}$. In particular,
we have $H^1(E,\mathcal{N}_{E/X}(-1))=0$. Similarly, we prove the
vanishing $H^1(C,\mathcal{N}_{C/X}(-1))=0$ for a general
arithmetically Cohen-Macaulay curve $C$ of genus $10$ and degree
$12$ on a general cubic threefold $X \subset \mathbb{P}^4$.
Finally, we prove the existence of a nonempty open subset of
triples $C,E,X$ as above which satisfy in addition $E \cap C=
\emptyset$,
$\mathrm{Ext}^1_{\mathcal{O}_X}(\mathcal{I}_{E/X}(2),\mathcal{O}_X)$
is $1$-dimensional and for the nontrivial extension $\mathcal{F}$
the vanishing $H^1(\mathcal{F}\otimes
\mathcal{I}_{C/X})=H^2(\mathcal{F}\otimes \mathcal{I}_{C/X})=0$
holds. The proofs are based on a computation in \textit{Macaulay2}
over a finite field and semi-continuity.
\end{abstract}

\subsection{Quintic elliptic curves}
\begin{theorem}\label{elliptic}
Let $E\subset X \subset \mathbb{P}^4$ be a general pair of an
elliptic normal curve on a general cubic threefold over an
algebraically closed field of characteristic $0$. Then the twisted
normal bundle of $E$ in $X$ splits as $\mathcal{N}_{E/X}(-1)= L
\oplus L^{-1}$ with $L \in \mathrm{Pic}^0(E)$, $L\not\cong
\mathcal{O}_E$. In particular, $H^1(\mathcal{N}_{E/X}(-1))=0$.
\end{theorem}
\begin{proof} First, we check the corresponding statement for a general pair $E\subset X\subset \mathbb{P}^4$ defined over a finite field $\mathbb{F}_p$ by computation in \textit{Macaulay2}. Initialization:~\\

{\small
\begin{verbatim}
     i1: p=101 -- a fairly small prime number
         Fp=ZZ/p -- a finite ground field
         R=Fp[x_0..x_4] -- coordinate ring of P^4
         setRandomSeed("beta")
\end{verbatim}}~\\

We start by randomly choosing a smooth cubic threefold $X$ and a
smooth quintic elliptic curve $E$ on it.
~\\
{\small
\begin{verbatim}
    i2 : m1=random(R^6,R^{6:-1});
         m=m1-transpose m1;
            -- a random skew symmetric 6x6 matrix of linear forms
         I=pfaffians(4,m_{0..4}^{0..4});
            -- the ideal an elliptic normal curve E
         singE=minors(codim I,jacobian I)+I;
         (codim I==3, degree I==5, genus I==1, codim singE==5)

    o2 = (true, true, true, true)

    i3 : f=pfaffians(6,m) -- ideal of X
         singf=ideal jacobian f;
         (codim f==1, degree f==3, codim singf == 5)

    o3 = (true, true, true)
\end{verbatim}}~\\

\noindent
Next, we compute the normal bundle and the first values of its Hilbert function:~\\
{\small
\begin{verbatim}
    i4 : I2=saturate(I^2+f);
         coN=prune (image( gens I)/ image gens I2);
             -- a module whose sheafication is the conormal sheaf
             -- of E in X
          N=Hom(coN,R^1/I); -- the module of global sections
                              -- of the normal bundle
          apply(toList(-1..2),i->hilbertFunction(i,N))

     o4 = {0, 10, 20, 30}
\end{verbatim}} ~\\
Hence, $\mathcal{N}_{E/X}(-1)$ has no sections, and since $\det
\mathcal{N}_{E/X}(-1) \cong \mathcal O_E$ has degree 0, we have
$H^1(\mathcal{N}_{E/X}(-1))=0$ as well. There are two
possibilities for the rank 2 vector bundle $\mathcal{N}_{E/X}(-1)$
according to the Atiyah classification \cite{Atiyah}. Either
$$\mathcal{N}_{E/X}(-1) \cong L_1\oplus L_2$$
with $L_2 \cong L_1^{-1} \in \mathrm{Pic}^0(E)$ or
$\mathcal{N}_{E/X}(-1)$ is an extension
$$ 0 \to L \to \mathcal{N}_{E/X}(-1)\to L \to 0$$
 with $ L \in \mathrm{Pic}^0(E)$  is 2-torsion. We check that we are in the first case:

{\small
\begin{verbatim}
     i5 : Nminus1 = N**R^{-1};
          time betti(EndN=Hom(Nminus1,Nminus1))

                  0  1
     o5 = total: 12 40
              0:  2  .
              1: 10 40
\end{verbatim}} ~\\

\noindent Thus, $H^0(\mathcal End(\mathcal N_{E/X}(-1)))$ is
two-dimensional. We compute the characteristic polynomial and the
eigenvalues of this pencil of endomorphisms. The command
SetRandomSeed("beta") above was chosen such that the
characteristic polynomial decomposes  completely over $\mathbb
F_p$ in this step of the computation.
~\\
{\small
\begin{verbatim}
     i6 : h0=homomorphism EndN_{0};
          h0a=map(R^10,R^10,h0)
          h1=homomorphism EndN_{1};
          h1a=map(R^10,R^10,h1) -- the corresponding matrices

     i7 : T=Fp[t] -- an extra ring
          chiA=det(sub(h0a,T)-t*sub(h1a,T));
              -- the characteristic polynomial
          chiAFactors = factor chiA



                  5        5
     o7 = (t - 47) (t - 14)


     i8 : -- We compute the eigenvalues and eigenspaces
          eigenValues=apply(2,c-> -((chiAFactors#c)#0)%ideal t)
          betti (N1=syz(h0a-eigenValues_0*h1a)
          betti (N2=syz(h0a-eigenValues_1*h1a)) -- the eigenspaces
          betti N
          L1=prune coker(presentation N|N1)**R^{-1};
          L2=prune coker(presentation N|N2)**R^{-1};
              -- the corresponding line bundles
          betti res L1 -- L1 (and L2) has a linear resolution

                 0  1  2 3
     o8 = total: 5 15 15 5
              1: 5 15 15 5
\end{verbatim}} ~\\
\noindent Finally, we check that $L1\oplus L2 \cong \mathcal
N_{E/X}(-1)$. {\small
\begin{verbatim}
     i9 : time betti (homL1L2=Hom(L1**L2,R^1/I))  -- used 32.25 seconds
              -- => L1 tensor L2 = O_E
           annihilator homL1L2==I -- check

     o9 = true

     i10 : time betti(iso=Hom(L1++L2,N))  -- used 9.22 seconds
           iso0=homomorphism iso_{0}
           iso1=homomorphism iso_{1}

     o10 = | 0 0 0 0 0 10  42  31  7   -9  |
           | 0 0 0 0 0 16  -27 -30 -21 -35 |
           | 0 0 0 0 0 6   -13 -19 -5  -29 |
           | 0 0 0 0 0 38  9   41  22  -30 |
           | 0 0 0 0 0 -3  -9  34  -31 1   |
           | 0 0 0 0 0 20  -4  -19 -5  6   |
           | 0 0 0 0 0 17  -2  -37 -6  -19 |
           | 0 0 0 0 0 -46 -18 -31 -26 -20 |
           | 0 0 0 0 0 43  -23 -47 -33 -43 |
           | 0 0 0 0 0 34  41  -35 -13 1   |

     i11 : det map(R^10,R^10,iso0+iso1)=!=0 -- N(-1) is isomorhic to L1++L2

     o11 = true

     i12 : prune ker(iso0+iso1)==0 and prune coker(iso0+iso1)==0
               -- kernel and cokernel are zero

     o12 = true
\end{verbatim}} ~\\
Since $L_1 \in \mathrm{Pic}^0(E)(\mathbb{F}_p)$ it has finite
order. We compute the order, just for fun, in the most naive way.
If the prime $p$ is larger a better method is necessary.
~\\
{\small
\begin{verbatim}
     i13 : time betti(twoL1=prune Hom(L2,L1))
           k=2;
           L1=twoL1;
           time while (rank target gens kL1=!=1) do (k=k+1;
               kL1=prune Hom(L2,kL1)); -- used 29 seconds
                                       -- in a case where the order k=52
           k -- the order of L1 in Pic E.

     o13 = 52

     i15 : betti kL1;
           kL1==R^1/I

     o15 = true
\end{verbatim}}~\\

To conclude from these computations the desired result in
characteristic zero, we argue that the computation above can be
seen as the reduction mod $p$ of computation over $\mathbb Z$. By
semi-continuity the vanishing
$$H^0(\mathcal{N}_{E_{\mathbb Q}/X_{\mathbb Q}}(-1))=H^1(\mathcal{N}_{E_{\mathbb Q}/X_{\mathbb Q}}(-1))=0$$
holds for the corresponding pair $(E_{\mathbb Q},X_{\mathbb Q})$
defined over $\mathbb Q$ as well. The splitting into line bundles
will  be defined over a quadratic extension field $K$ of $\mathbb
Q$ and the line bundle most likely will have infinite order in
$\mathrm{Pic}^0(E_{\mathbb Q})(K)$.
\end{proof}

\begin{remark} By the computation above we know
from the example $E \subset X$ defined over an open part
$\mathrm{Spec}\, \mathbb Z$, that the same result holds for
algebraically closed fields of positive characteristic except for
possible finitely many primes $p$. In principle, one could try to
compute these exceptional primes by computing an example over
$\mathbb Q$, and then try to verify the result one by one for the
remaining  primes. We believe that this is currently
computationally out of reach.
\end{remark}

\subsection{ACM curves of genus 10 and degree 12}

In this section we prove the following

\begin{theorem}\label{main}
The space of pairs $C\subset X \subset {\mathbb P}^4$ of smooth
arithmetically Cohen-Macaulay curves $C$ of degree $12$ and genus
$10$ on a cubic threefold $X$ is unirational and dominates the
moduli space $\mathcal{M}_{10}$ of curves of genus $10$ and the
Hilbert scheme of cubic threefolds in ${\mathbb P}^4$ with the
maps defined over $\mathbb Q$. Moreover, for a general pair $C
\subset X$ the following holds:
\begin{enumerate}

\item The line bundle
$ \mathcal{O}_C(1)$ is a smooth isolated point of the
Brill-Noether space $W^4_{12}(C)\subset \mathrm{Pic}^{14}(C)$.

\item
The module of global sections $\sum_{n\in \mathbb Z}
H^0(\omega_C(n))$ of the dualizing sheaf $\omega_C$ is generated
by its two sections in degree $-1$ as an $S=\sum_{n \in \mathbb Z}
H^0(\mathbb P^4,\mathcal O(n))$-module.

\item
The twisted normal bundle of $C$ in $X$ satisfies
$h^1(\mathcal{N}_{C/X}(-1))=0$.
\end{enumerate}
\end{theorem}

As in section 1, we will prove the result by a computation over a
finite field and semi-continuity. The key ingredient is the
following unirational construction of the desired curves. Suppose
$C$ is a smooth projective curve of genus $10$ defined over a
field $k$ together with line bundles $L_1$, $L_2$ with $|L_1|$ a
$\mathfrak{g}^1_6$ and $|L_2|$ a $\mathfrak{g}^2_9$. Let $C'$
denote the image under the map
\begin{equation*}
C\xrightarrow{|L_1|,|L_2|} \mathbb{P}H^0(C,L_1)\times
\mathbb{P}H^0(C,L_2)=\mathbb{P}^1\times \mathbb{P}^2.
\end{equation*}
We say that $C$ is of maximal rank if the map $H^0
\mathcal{O}_{\PP^2}(n,m) \rightarrow H^0(L_1^{\otimes n} \otimes
L_2^{\otimes m})$ is of maximal rank for all $n,m\geq 1$. Under
the assumption of maximal rank of $C$ the image $C'$ is isomorphic
to $C$ and the Hilbert series of the truncated vanishing ideal
\begin{equation*}
I_{\mathrm{trunc}} = \bigoplus_{n\geq 3, m\geq 3}
H^0(\mathcal{I}_{C'}(n,m))
\end{equation*}
in the Cox-Ring $S=k[x_0,x_1,y_0,y_1,y_2]$ of $\mathbb{P}^1\times
\mathbb{P}^2$ is
\begin{equation*}
H_{I_{\mathrm{trunc}}}(s,t) =  \frac{3s^4t^5-6s^4t^4-3s^3t^5+3s^3t^4+4s^3t^3}{(1-s)^2(1-t)^3}.\\
\end{equation*}
In other words, we expect a bigraded free resolution of type
\begin{equation*} \label{res}
0 \rightarrow F_2 \rightarrow F_1 \rightarrow F_0 \rightarrow
I_{\mathrm{trunc}} \rightarrow 0
\end{equation*}
with modules $F_0=S(-3,-3)^4 \oplus S(-3,-4)^3$, $F_1=S(-3,-5)^3\oplus S(-4,-4)^6$ and $F_2=S(-4,-5)^3$.\\
Turning things around, we find the following unirational
construction for such curves: For a general map $M:F_2\rightarrow
F_1$  let $K$ be the cokernel of the dual map
$M^*:F_1^*\rightarrow F_2^*$. For the first terms of a minimal
free resolution of $K$ we expect
\begin{equation*}
 \ldots \rightarrow G \xrightarrow{N'} F_1^* \rightarrow F_2^* \rightarrow K \rightarrow 0
\end{equation*}
with $G=S(2,4)^3 \oplus S(3,3)^9 \oplus S(3,4)^3 \oplus S(4,2)^6$
. Composing $N'$ with a general map $F_0^*\rightarrow G$ and
dualizing again yields a map $N:F_1 \rightarrow F_0$. Finally,
$\ker(F_0^* \xrightarrow{N^*} F_1^*)\cong S$ and the entries of
the matrix $S \rightarrow F_0^*$ generate $I_{\mathrm{trunc}}$.
The following Code for \textit{Macaulay2} \cite{M2} realizes this
construction over an arbitrary field, here in particular for
random choices over a finite field $\mathbb{F}_p$:
~\\~\\
{\small
\begin{verbatim}
     i1 : setRandomSeed"I am feeling lucky"; -- initiate random generator
          p=32009; -- a prime number
          Fp=ZZ/p; -- a prime field
          S=Fp[x_0,x_1,y_0..y_2, Degrees=>{2:{1,0},3:{0,1}}];
              -- Cox ring of P^1 x P^2
          m=ideal basis({1,1},S); -- irrelevant ideal

     i2 : randomCurveGenus10Withg16=(S)->(
          M:=random(S^{6:{-4,-4},3:{-3,-5}},S^{3:{-4,-5}});
              -- a random map F1 <--M-- F2
          N':=syz transpose M; -- syzygy-matrix of the dual of M
          N:=transpose(N'*random(source N',S^{3:{3,4},4:{3,3}}));
          ideal syz transpose N) -- the vanishing ideal of the curve

     i3 : IC'=saturate(randomCurveGenus10Withg16(S),m);
\end{verbatim}}
~\\
As being of maximal rank is an open condition this computation proves the existence of a nonempty unirational component $H$ in the Hilbert scheme  $\mathrm{Hilb}_{(6,9),10}(\mathbb{P}^1\times \mathbb{P}^2)$ of curves  of bidegree $(6,9)$ and genus $10$.\\

By semi-continuity we get the first half of the following Theorem.

\begin{theorem}
The Hilbert scheme $\mathrm{Hilb}_{(6,9),10}(\mathbb{P}^1\times
\mathbb{P}^2)$ has a unirational component $H$ over $\mathbb Q$
that dominates the moduli space $\mathcal{M}_{10}$. \end{theorem}

\begin{proof}
The main missing ingredient is to prove that in our example above
the line bundles $L_1$ and $L_2$ will be behave like general line
bundles in $W^1_6(C)$ and $W^2_9(C)$ for a general curve $C$.
Recall the following facts from Brill-Noether theory \cite{ACGH}:
For a general smooth curve $C$ of genus $g$ the Brill-Noether loci
\begin{equation*}
W^r_d(C)=\{ L \in \mathrm{Pic}^d(C)\ |\ h^0(L)\geq r+1 \}
\end{equation*}
are non-empty and smooth away from $W^{r+1}_d(C)$ of dimension
$\rho$ if and only if the Brill-Noether number
\begin{equation*}
\rho = \rho(g,r,d) = g-(r+1)(g-d+r) \geq 0.
\end{equation*}
Moreover, $W^r_d(C)$ is connected if $\rho>0$ and the tangent
space at a linear series $L \in W^r_d(C) \smallsetminus
W^{r+1}_d(C)$ is the dual of the cokernel of the Petri-map
\begin{equation*}
H^0(C,L) \otimes H^0(C,\omega_C\otimes L^{-1}) \rightarrow
H^0(C,\omega_C).
\end{equation*}
Now let $\eta:C \rightarrow C'$ be a normalization of our given point $C' \in H$. $\eta$ will be an isomorphism, but we do not know this yet. We can check computationally that the linear systems $L_1 =\eta^*\mathcal{O}_{\mathbb{P}^1}(1)$ and $L_2=\eta^*\mathcal{O}_{\mathbb{P}^2}(1)$ are smooth points in the respective $W^{r_i}_{d_i}(C)$ for $i=1,2$:\\
In order to check $L_2$, we start by computing the plane model $\Gamma\subset \PP^2$ of $C'$: ~\\

{\small
\begin{verbatim}
     i4 : Sel=Fp[x_0,x_1,y_0..y_2,MonomialOrder=>Eliminate 2];
              -- eliminination order
          R=Fp[y_0..y_2]; -- coordinate ring of P^2
          IGammaC=sub(ideal selectInSubring(1,gens gb sub(IC',Sel)),R);
              -- ideal of the plane model
\end{verbatim}}~\\
We check that $\Gamma$ is a curve of desired degree and genus and its singular locus $\Delta$ consists only of ordinary double points:~\\

{\small
\begin{verbatim}
     i5 : distinctPoints=(J)->(
             singJ:=minors(2,jacobian J)+J;
             codim singJ==3)

     i6 : IDelta=ideal jacobian IGammaC + IGammaC; -- singular locus
          distinctPoints(IDelta)

     o6 = true

     i7 : delta=degree IDelta;
          d=degree IGammaC;
          g=binomial(d-1,2)-delta;
          (d,g,delta)==(9,10,18)

     o7 = true
\end{verbatim}}~\\
We compute the free resolution of $I_{\Delta}$:~\\
{\small
\begin{verbatim}
     i8 : IDelta=saturate IDelta;
          betti res IDelta

                 0 1 2
     o8 = total: 1 4 3
              0: 1 . .
              1: . . .
              2: . . .
              3: . . .
              4: . 3 .
              5: . 1 3
\end{verbatim}}~\\
(We can deduce that $\Gamma$ is irreducible from this information:
Suppose $\Gamma$ decomposes in two parts of degree $a$ and $b$
with $a+b=9$ and, say, $a< b$ then the intersection points of two
components would be among the points of $\Delta$. The cases
$(a,b)=(1,8)$ and $(2,7)$ are excluded because $I_\Delta$ is
generated by sextics, $(4,5)$ is excluded because $20 > 18$ and
$(3,6)$ is excluded because $\Delta$ is not a complete
intersection. Thus $C$ the normalization of $\Gamma$ is isomorphic
to a smooth irreducible curve of genus $g=10$, and $C'$ is smooth
because $10=g \le p_a C'  \le 10$.)

From Riemann-Roch we deduce $h^0(C,L_2)=3$ since
$h^1(C,L_2)=h^0(C,\omega_C\otimes L_2^{-1})=h^0({\mathbb
P}^2,\mathcal{I}_\Delta(5))=3$. The Petri map for $L_2$ can be
identified with
\begin{equation*}
H^0(\mathbb{P}^2, \mathcal{I}_{\Delta}(d-4)) \otimes
H^0(\mathbb{P}^2,\mathcal{O}_{\mathbb{P}^2}(1)) \rightarrow
H^0(\mathbb{P}^2,\mathcal{I}_{\Delta}(d-3)).
\end{equation*}
This map is injective since there are no linear relations among
the three quintic generators of $I_{\Delta}$. So $L_2 \in
W^3_9(C)$ is a smooth point of dimension $\rho_2=1$. \medskip

Turning to $L_1$, we compute the embedding $C \rightarrow {\mathbb P} H^0(C,\omega_C\otimes L_1^{-1}) = {\mathbb P}^4$ as follows~\\
{\small
\begin{verbatim}
     i9 : LK=(mingens IDelta)*random(source mingens IDelta, R^{10:{-6}});
              -- compute a basis of the Riemann-Roch space L(Omega_C)
          Pt=random(Fp^1,Fp^2); -- random point in P^1
          L1=substitute(IC',Pt|vars R); -- L1 is the fiber over Pt
          KD=LK*(syz(LK % gens L1))_{0..4};
              -- compute a basis of those elements in L(Omega_C) that
              -- vanish in L1
          T=Fp[z_0..z_4]; -- coordinate ring of P^4
          phiKD=map(R,T,KD); -- embedding
          IC=preimage_phiKD(IGammaC);
          degree IC==12 and genus IC==10
     o9 = true

     i10 : betti(FC=res IC)

                 0 1 2 3
     o10 = total: 1 8 9 2
               0: 1 . . .
               1: . . . .
               2: . 8 9 .
               3: . . . 2
\end{verbatim}}~\\
From the length of the resolution $F_C$ we see that the image of
$C$ in $\mathbb P^4$ is arithmetically Cohen-Macaulay. The dual
complex $\mathrm{Hom}_S(F_C,S(-5))$ is a resolution of
$\bigoplus_{n\in \mathbb Z} H^0(\omega_C(n))$. Thus this module is
generated by its two sections in degree $-1$ and
$h^0(L_1)=h^0(C,\omega_C(-1))=2$. The Petri map can be identified
with
\begin{equation*}
H^0(C,\omega_C(-1))  \otimes
H^0(\mathbb{P}^4,\mathcal{O}_{\mathbb{P}^4}(1)) \rightarrow
H^0(C,\omega_C).
\end{equation*}
Here, this map is an isomorphism, because there is no linear
relation among the two generators, and $L_1$ is a smooth isolated
point in $W^1_6(C)$.
Thus our random example over the finite field is as expected, and semi-continuity proves that the same is true for the triple $(C, L_1,L_2)$ defined over an open part of $\mathrm{Spec} \, \mathbb Z$ whose reduction mod $p$ is the given randomly selected curve.\\

The map $H \rightarrow \mathcal{M}_{10}$ factors over
$Z=\mathcal{W}^1_6 \times_{\mathcal{M}_{10}} \mathcal{W}^2_{9}$
and the  fiber of $H\rightarrow Z$ for a triple $(C,L_1,L_2)$
(without automorphisms) with $h^0(C,L_1)=2$ and $h^0(C,L_2)=3$
 is $\mathrm{PGL}(2)\times \mathrm{PGL}(3)$.
 The fiber dimension of $Z\rightarrow \mathcal{M}_g$ is $\rho_1+\rho_2=0+1=1$, as expected.
 \end{proof}

\begin{remark}
Note that $H$ dominates the Severi variety $V_{9,10}$ of reduced and irreducible plane curves of degree $9$ and genus $10$ as well as the Hurwitz scheme $H_{6,10}$ of $6$-gonal curves of genus $10$. \\
In fact, the outlined approach allows to prove the existence of
unirational components of
$\mathrm{Hilb}_{(d_1,d_2),g}(\mathbb{P}^1 \times \mathbb{P}^2)$
for several values $(d_1,d_2,g)$. Particularly, we find that the
Severi variety $V_{10,11}$ and the Hurwitz schemes $H_{6,g}$ of
$6$-gonal covers for $g\leq 40$ are unirational. The last
statement is proved using liaison in $\mathbb{P}^1\times
\mathbb{P}^2$, see \cite{Geiss}.
\end{remark}

\noindent {\it Proof of Theorem \ref{main}.} We are nearly done.
The embedding of
$$C \hookrightarrow \mathbb P H^0(C,\omega_C \otimes L_1^{-1}) \cong \mathbb P^4$$
is a curve which satisfies (1) and (2). Since $L_1$ and
equivalently $\mathcal O_C(1) \in W^4_{12}(C)$ is Petri general
this proves the existence of a unirational component
$$H' \subset \mathrm{Hilb}_{12t+1-10}(\mathbb{P}^4).$$  Since the Hurwitz scheme $H_{6,10}$ is irreducible,
we can conclude that the induced rational map
$H'//\mathrm{PGL}(5)\rightarrow \mathcal{M}_{10}$ is generically
finite of degree
\begin{equation*}
g! \prod_{i=0}^r \frac{i!}{(g-d+r+i)!} = 42 = \deg W^1_6(C)
\end{equation*}
as is well-known (\cite{ACGH}, Ch.V). Choosing a cubic threefold
containing $C$ is the same as choosing a point in the projective
space $\mathbb{P}H^0(\mathbb{P}^4,\mathcal{I}_C(3))$. Hence,
\begin{equation*}
V=\{ (C,X) \ |\ C\in H' \hbox{ ACM and }X \in
\mathbb{P}H^0(\mathbb{P}^4,\mathcal{I}_C(3)) \hbox{ smooth }\}
\end{equation*}
is unirational as well. For a random pair $(C,X)\in V$ we compute the normal sheaf $\mathcal{N}_{C/X}$ of $C$ in $X$ and check that $h^i(\mathcal{N}_{C/X}(-1))=0$ for $i=0,1$:~\\
{\small
\begin{verbatim}
     i11 : IX=ideal((mingens IC)*random(source mingens IC,T^{1:-3}));
           IC2=saturate(IC^2+X);
           cNCX=image gens IC/ image gens IC2; -- the conormal sheaf in X
           NCX=sheaf Hom(cNCX,T^1/IC); -- the normal sheaf in X

     i12 : HH^0 NCX(-1)==0 and HH^1 NCX(-1)==0

     o12 = true

     i14 : HH^0 NCX==Fp^24 and HH^1 NCX==0

     o14 = true
 \end{verbatim}}~\\
 With a similar computation for $\mathcal{N}_{C/\mathbb{P}^4}$ we check that $H'$ is a generically smooth component of the Hilbert scheme $\mathrm{Hilb}_{12t+1-10}(\mathbb{P}^4)$ of expected dimension $51$ and $C$ is a smooth point in $H'$.
~\\
{\small
\begin{verbatim}
      i15 : cNCP= prune(image (gens IC)/ image gens saturate(IC^2));
            NCP=sheaf Hom(cNCP,T^1/IC);
            HH^1 (NCP)==0 and HH^0 (NCP)==Fp^51
      o15 = true
 \end{verbatim}}~\\
 Consider the maps
 \begin{equation*}
 \xymatrix{ & V \ar[dr]^{\pi_2} \ar[dl]_{\pi_1} &  \\
            H' && \mathbb{P}H^0(\mathbb{P}^4,\mathcal{O}_{\mathbb{P}^4}(3)) \cong \mathbb P^{34}}
 \end{equation*}
 The fibre of $\pi_1$ over a point $C$ is exactly ${\mathbb P} H^0({\mathbb P}^4,\mathcal{I}_C(3)) \cong \mathbb P^7 $, hence $V$ is irreducible of dimension $58$.
 The map $\pi_2$ is smooth of dimension  $h^0(C, \mathcal{N}_{C/X})=24 $ at $(C,X)$. Thus $\pi_2$ is surjective.
 By semi-continuity the desired vanishing holds for the general curve on a general cubic.
\qed \medskip
\newpage
\subsection{Cohomology of Extensions}

In order to prove Corollary \ref{finalCorollary} for arbitrary
rank the following statement is needed:
\begin{proposition} \label{cohomologyExtension}
Let $k$ be an algebraically closed field of characteristic $0$.
There is an open subset $U$ of the space of triples $C,E \subset
X$ with $C$ and ACM curve of genus $10$ and degree $12$, $E$ an
elliptic normal curve of degree $5$ not meeting $C$ and $X$ a
smooth cubic threefold over $k$ with the following properties:
\begin{enumerate}
\item $U$ dominates the space $\mathbb{P}H^0(\mathbb{P}^4,\mathcal{O}_{\mathbb{P}^4}(3))$ of cubic threefolds and the spaces of pairs $E\subset X$ and $C\subset X$. In particular the pair $E \subset X$ and the pair $C \subset X$ satisfy all assertions of Theorem \ref{elliptic} and \ref{main} respectively.
\item For every triple $C,E\subset X$ in $U$ the extension group $\Ext_{\mathcal{O}_X}^1(\mathcal{I}_{E/X}(2), \mathcal{O}_X))$ is $1$-dimensional and for the non-trivial extension
$$
0 \to \mathcal{O}_X \to \mathcal{F} \to \mathcal{I}_{E/X}(2) \to 0
$$
we have the vanishing  $H^1(\mathcal{F}\otimes
\mathcal{I}_{C/X})=H^2(\mathcal{F}\otimes \mathcal{I}_{C/X})=0$.
\end{enumerate}
\end{proposition}

\begin{proof} Again, our strategy is to construct a triple $C,E \subset X$ over a finite field with the help of \textit{Macaulay2} and then establishing the theorem in characteristic $0$ with semi-continuity. \\
The bottleneck of this approach is to construct $E$ and $C$ such
that there is a cubic threefold which contains both curves.
Since $H^0(\mathcal{O}_{\PP^4}(3))$ is $35$-dimensional, for a general pair $(E,C)$ the $20$-dimensional subspace $W_E=H^0(\mathcal{I}_{E/\PP^4}(3))$ and the $8$-dimensional subspace $W_C=H^0(\mathcal{I}_{E/\PP^4}(3))$ will have a trivial intersection. More precisely, the locus $M$ of pairs $(E,C)$ with $W_E \cap W_C \neq 0$ has expected codimension $8$ in $H=H_1 \times H_2 \subset \mathrm{Hilb}_{5t}(\PP^4) \times \mathrm{Hilb}_{12t-9}(\PP^4)$ where $H_1$ is the subscheme whose points correspond to smooth elliptic normal curves and $H_2$ the subscheme whose points correspond to smooth ACM curves. \\
One way to find points in $M$ is by searching over a small finite
field. Heuristically, the probability for a random point $(E,C)
\in H( \mathbb{F}_p)$ to lie in $M(\mathbb{F}_p)$ is
$$ \frac{\# M(\mathbb{F}_p)}{\# H(\mathbb{F}_p)} \approx \frac{1}{p^{8}}.$$
From the Weil formula we see that this approximation is
asymptotically correct as $p\rightarrow \infty$. Practice shows
that it is a reasonable heuristic even for small $p$ in many
cases. However, over very small fields ($p=2,3$) most
hypersurfaces are singular, see \cite{BS}. Hence we must not
choose $p$ too small in order to minimize the total runtime.
Empirically, $p=5$ is a good choice. Turning to the construction,
we start with a random smooth arithmetically Cohen-Macaulay curve
$C$ of genus $10$ and degree $12$ in $\PP^4$. To keep things clear
we capsulated the construction of the preceding section in a
function that returns the vanishing ideal of such a curve:
~\\
{\small
\begin{verbatim}
     i1 : load"UlrichBundlesOnCubicThreefolds.m2";
          Fp=ZZ/5;
          T=Fp[z_0..z_4];
          setRandomSeed("gamma");

     i2 : time IC=randomCurveGenus10Degree12(T);
             -- used 32.4096 seconds
\end{verbatim}}
For the sake of replicability we write down the curve used in our
example. To do this is in a space-saving way, we write down the
$9\times8$ matrix $m_C$ with linear entries in the free resolution
of $I_C$. From $m_C$ the curve can easily be regained: {\small
\begin{verbatim}
     i3 : mC= transpose((res IC).dd_2);

     i4 : IC==ideal syz mC
            -- regain the curve from mC

     o4 : true
\end{verbatim}}~\\
In our example, we have {\tiny
\begin{flushleft}
$m_C=\left(\begin{array}{cccc}
    {-2 {z}_{3}}&
      2 {z}_{1}-2 {z}_{3}-{z}_{4}&
      -2 {z}_{0}-{z}_{1}-2 {z}_{3}+2 {z}_{4}&
      {z}_{0}+{z}_{3}-2 {z}_{4}\\
      2 {z}_{1}-{z}_{2}+{z}_{3}-2 {z}_{4}&
      -2 {z}_{0}+2 {z}_{1}+2 {z}_{4}&
      -2 {z}_{0}+2 {z}_{4}&
      {z}_{3}-{z}_{4}\\
      -2 {z}_{1}+{z}_{3}&
      2 {z}_{0}-{z}_{1}-{z}_{2}-2 {z}_{3}&
      {z}_{0}+2 {z}_{3}+{z}_{4}&
      2 {z}_{3}-{z}_{4}\\
      -{z}_{1}-{z}_{3}-{z}_{4}&
      {z}_{0}-{z}_{1}-2 {z}_{3}&
      {z}_{0}-{z}_{3}-{z}_{4}&
      -2 {z}_{3}+2 {z}_{4}\\
      2 {z}_{3}&
      2 {z}_{1}+{z}_{3}+{z}_{4}&
      -2 {z}_{0}-{z}_{2}-2 {z}_{4}&
      {z}_{3}-{z}_{4}\\
      -{z}_{1}-{z}_{3}+2 {z}_{4}&
      {z}_{0}+{z}_{3}+2 {z}_{4}&
      2 {z}_{3}&
      {-2 {z}_{3}}\\
      {z}_{1}+2 {z}_{3}&
      -{z}_{0}-2 {z}_{1}+2 {z}_{3}+{z}_{4}&
      2 {z}_{0}-2 {z}_{3}-2 {z}_{4}&
      {-{z}_{2}}\\
      2 {z}_{1}+2 {z}_{3}-{z}_{4}&
      -2 {z}_{0}-2 {z}_{1}-2 {z}_{4}&
      2 {z}_{0}+2 {z}_{3}-{z}_{4}&
      {z}_{4}\\
      {z}_{1}-2 {z}_{3}&
      -{z}_{0}-{z}_{1}+{z}_{3}-{z}_{4}&
      {z}_{0}-{z}_{3}-2 {z}_{4}&
      -{z}_{3}+2 {z}_{4}\\     \end{array} \right.$
      \end{flushleft}
\begin{flushright}
$\left.\begin{array}{cccc}
 {-2 {z}_{3}}&
      {z}_{2}-2 {z}_{3}-{z}_{4}&
      {-2 {z}_{4}}&
      -{z}_{2}+{z}_{3}-{z}_{4}\\
      {z}_{0}-{z}_{3}&
      2 {z}_{2}+2 {z}_{3}-{z}_{4}&
      2 {z}_{2}-{z}_{3}+2 {z}_{4}&
      -{z}_{2}+{z}_{3}-2 {z}_{4}\\
      {z}_{1}+2 {z}_{4}&
      -2 {z}_{2}-2 {z}_{3}-{z}_{4}&
      -{z}_{2}-{z}_{3}-2 {z}_{4}&
      2 {z}_{2}+2 {z}_{3}\\
      -{z}_{1}-{z}_{3}&
      {z}_{0}+2 {z}_{2}+{z}_{3}+2 {z}_{4}&
      {z}_{3}+2 {z}_{4}&
      {z}_{2}+2 {z}_{3}+2 {z}_{4}\\
      {z}_{3}&
      {z}_{1}+2 {z}_{2}+{z}_{3}&
      -2 {z}_{3}-{z}_{4}&
      {z}_{3}+{z}_{4}\\
      -{z}_{3}+2 {z}_{4}&
      -{z}_{1}-2 {z}_{2}+2 {z}_{3}-{z}_{4}&
      {z}_{0}-{z}_{4}&
      2 {z}_{3}-2 {z}_{4}\\
      0&
      {z}_{2}+{z}_{3}&
      {z}_{1}+{z}_{2}+2 {z}_{3}-{z}_{4}&
      -{z}_{2}-{z}_{3}+2 {z}_{4}\\
      -{z}_{2}-{z}_{3}-2 {z}_{4}&
      2 {z}_{2}&
      -{z}_{2}+{z}_{3}+{z}_{4}&
      {z}_{0}-2 {z}_{2}+2 {z}_{3}+2 {z}_{4}\\
      -{z}_{3}-{z}_{4}&
      {z}_{2}-{z}_{3}&
      {z}_{3}-2 {z}_{4}&
      {z}_{1}-{z}_{2}+{z}_{4}\\
       \end{array} \right).$
\end{flushright}}
In the next step we search for an elliptic normal curve $E$ such
that $C$ and $E$ lie on a common cubic threefold $X$. Picking $E$
at random and checking whether there is a relation between the
generators of $H^0(\mathcal{I}_{C/\mathbb{P}^4}(3))$ and
$H^0(\mathcal{I}_{E/\mathbb{P}^4}(3))$ takes about $0.01$ seconds
a time on a $2.4\ \mathrm{GHz}$ processor. Hence we expect to find
a such an $E$ within a span of about one hour. \pagebreak {\small
\begin{verbatim}
     i5 : getEllipticWithCommonThreefold=(IC)->(
             max3:=ideal basis(3,T);
                -- third power of the maximal ideal
             for attemptsHS from 1 do (
                mEtmp:=random(T^5,T^{5:-1});
                mE:=mEtmp-transpose mEtmp;
                   -- the 5x5 skew-symmetric matrix
                IE:=pfaffians(4,mE);
                   -- the elliptic curve E
                if rank source gens intersect(IE+IC,max3)<28 then (
                   rltn:=(syz(gens IC|gens intersect(IE,max3)))_{0};
                      -- the relation between the generators
                   X:=ideal (gens IC*rltn^{0..7});
                      -- the cubic threefold
                   <<"attempts hypersurface = "<<attemptsHS;
                      -- print number of attempts
                   return(mE,X))))
\end{verbatim}}
We also have to check that the cubic hypersurface $X$ is smooth
and that the twisted normal bundles $\mathcal{N}_{E/X}(-1)$ has no
global sections, as expected. {\small
\begin{verbatim}
     i6 : normalSheaf=(I,X)->(
             I2:=saturate(I^2+X);
             cNIX:=image gens I/ image gens I2;
             sheaf Hom(cNIX,(ring I)^1/I))

     i7 : sectionsTwistedNormalBundle=(mE,X)->(
             IE:=pfaffians(4,mE);
             NEX:=normalSheaf(IE,X);
             HH^0(NEX(-1)))

\end{verbatim}}

Recall from \cite{Eis}, that $\mathcal{O}_E$ has an eventually
3-periodic free resolution as an $\mathcal{O}_X$-module
\begin{eqnarray*}
\ldots \xrightarrow{q} \mathcal{O}_X(-6)^6 \xrightarrow{m}
\mathcal{O}_X(-5)^6 \xrightarrow{q} \mathcal{O}_X(-3)^6
 \to \mathcal{O}_X(-2)^5  \to \mathcal{O}_X \to \mathcal{O}_E \to 0 \label{resElliptic}
 \end{eqnarray*}
 whose  higher syzygy modules are independent of choice of the  section $s \in H^0(\mathcal{F})$  defining $E$.
Thus the number of sections
$$ \mathcal{N}_{E/X}(-1) \cong \mathcal{I}_{E/X}/\mathcal{I}_{E/X}^2(1) \cong \mathcal{F} \otimes
\mathcal{O}_E(-1)$$ depends only on $\mathcal{F}$ but not on $s\in
H^0(\mathcal{F})$: Tensoring the perodic resolution with
$\mathcal{F}(-1)$ and the fact that $\mathcal{F}$ has no
intermediate cohomology yields
$$ H^0(\mathcal{N}_{E/X}(-1)) \cong \ker( H^3( \mathcal{F}(-1)) \otimes \mathcal{K}_3)\to H^3( \mathcal{F}^6(-4)),$$
and $\mathcal{K}_3 =\ker ( \mathcal{O}_X(-3)^6 \to
\mathcal{O}_X(-2)^5) \cong \im (\mathcal{O}_X(-4)^6
\xrightarrow{q} \mathcal{O}_X(-3)^6)$ is independent of $E$. So if
$H^0(\mathcal{N}_{E/X}(-1))= 0$ then for any other elliptic curve
$E'$ corresponding to a global section of $\mathcal{F}$ the
cohomology $H^0(\mathcal{N}_{E'/X}(-1))$ is also  vanishing.
Putting everything together, we have the following search routine:
{\small
\begin{verbatim}
     i8 : time for attemptsN from 1 do (
             time for attemptsS from 1 do (
                time (mE,X)=getEllipticWithCommonThreefold(IC);
                if isSmooth X then (
                   <<"attempts smooth = "<<attemptsS;
                   break));
             if sectionsTwistedNormalBundle(mE,X)==0 then (
                <<"attempts normalbundle = "<<attemptsN;
                break));

         -- the output:
         attempts hypersurface = 25831     -- used 221.619 seconds
         attempts hypersurface = 206719     -- used 1825.24 seconds
         attempts hypersurface = 132506     -- used 1154.79 seconds
         attempts smooth = 3     -- used 3201.66 seconds
         attempts normalbundle = 1     -- used 3202.02 seconds
\end{verbatim}}

The extension is given as the cokernel of $m$ which is accessable
through the resolution of $\mathcal{O}_E$: {\small
\begin{verbatim}
     i9 : m0=sub((res sub(pfaffians(4,mE),T/X)).dd_4,T);
            -- the matrix in the resolution
          baseChange=Hom(coker m0,coker transpose m0);
          b=map(T^6,T^6,homomorphism baseChange_{0});
            -- we compute a skewsymmetrization of m0
          m=b*m0;
          pfaffians(6,m)==X

     o9 = true
\end{verbatim}}
In our example, we have

{\tiny $ m = \left(\begin{array}{ccc}
  0&-2 {z}_{1}-{z}_{2}+2 {z}_{4}&
      -2 {z}_{3}-{z}_{4}\\
      2 {z}_{1}+{z}_{2}-2 {z}_{4}&
      0&
      {z}_{0}+2 {z}_{2}-2 {z}_{3}+{z}_{4}\\
      2 {z}_{3}+{z}_{4}&
      -{z}_{0}-2 {z}_{2}+2 {z}_{3}-{z}_{4}&
      0\\
      2 {z}_{2}-2 {z}_{3}-2 {z}_{4}&
      -2 {z}_{0}+2 {z}_{1}-{z}_{2}+2 {z}_{3}-2 {z}_{4}&
      -2 {z}_{2}+2 {z}_{3}+{z}_{4}\\
      -{z}_{2}+{z}_{3}+{z}_{4}&
      2 {z}_{0}-{z}_{2}-{z}_{3}-2 {z}_{4}&
      {z}_{0}+2 {z}_{2}-{z}_{4}\\
      -2 {z}_{0}+{z}_{2}-{z}_{4}&
      {z}_{0}+2 {z}_{1}+{z}_{2}+2 {z}_{3}+2 {z}_{4}&
      -{z}_{1}-{z}_{2}+2 {z}_{3}+2 {z}_{4}\\
\end{array} \right. $
\begin{flushright}
$\left. \begin{array}{ccc} -2 {z}_{2}+2 {z}_{3}+2 {z}_{4}&
      {z}_{2}-{z}_{3}-{z}_{4}&
      2 {z}_{0}-{z}_{2}+{z}_{4}\\
      2 {z}_{0}-2 {z}_{1}+{z}_{2}-2 {z}_{3}+2 {z}_{4}&
      -2 {z}_{0}+{z}_{2}+{z}_{3}+2 {z}_{4}&
      -{z}_{0}-2 {z}_{1}-{z}_{2}-2 {z}_{3}-2 {z}_{4}\\
      2 {z}_{2}-2 {z}_{3}-{z}_{4}&
      -{z}_{0}-2 {z}_{2}+{z}_{4}&
      {z}_{1}+{z}_{2}-2 {z}_{3}-2 {z}_{4}\\
      0&
      2 {z}_{1}-{z}_{2}-{z}_{3}-2 {z}_{4}&
      -{z}_{0}+{z}_{1}+{z}_{2}+2 {z}_{4}\\
      -2 {z}_{1}+{z}_{2}+{z}_{3}+2 {z}_{4}&
      0&
      -{z}_{0}+{z}_{1}-2 {z}_{2}\\
      {z}_{0}-{z}_{1}-{z}_{2}-2 {z}_{4}&
      {z}_{0}-{z}_{1}+2 {z}_{2}&
      0\\
\end{array}\right).$
\end{flushright}}

A smooth random section $E'$ of the bundle $\mathcal{F}$ can also
be obtained very easily:

{\small
\begin{verbatim}
     i10 : IE'=for i from 1 do (
                  b:=random(T^6,T^6);
                  m':=b*m*transpose b;
                  IE':=pfaffians(4,m'_{0..4}^{0..4});
                  if isSmooth IE' and dim(IC+IE')==0 then break(IE'));
\end{verbatim}}

In order to check that $C$ and $E'$ are smooth points in
$\mathrm{Hilb}_{12t-9}(X)$ and $\mathrm{Hilb}_{5t}(X)$,
respectively, we compute the cohomology groups of the normal
sheaves: {\small
\begin{verbatim}
     i11 : NE'X=normalSheaf(IE',X);
           HH^0(NE'X)==Fp^10 and HH^1(NE'X)==0

     o11 = true

     i12 : NCX=normalSheaf(IC,X);
           HH^0(NCX)==Fp^24 and HH^1(NCX)==0

     o12 = true
\end{verbatim}}

Finally, we compute the cohomology groups of $\mathcal{F}\otimes
\mathcal{I}_{C/X}$: {\small
\begin{verbatim}
     i13 : M=coker sub(m,T/X);
              -- this is a module whose sheafification is an extension
           sheafMIC=sheaf(M)**sheaf(module sub(IC,T/X));
           HH^1 sheafMIC==0 and HH^2 sheafMIC==0

     o13 = true
\end{verbatim}}
\end{proof}

\end{document}